\documentclass[12pt,leqno]{amsart}
\usepackage{amsmath,amssymb,amsfonts,amsthm}
\usepackage{eucal,graphicx,colortbl,url,epsfig}
\usepackage[table]{xcolor}
\usepackage{rotating}

\newcommand \comment[1]{}			

\newtheorem{lem}{Lemma}[section]
\newtheorem{cor}[lem]{Corollary}
\newtheorem{prop}[lem]{Proposition}
\newtheorem{thm}[lem]{Theorem}
\newtheorem{conj}[lem]{Conjecture}

\newtheorem*{conje}{Conjecture}
\newtheorem*{thrm}{Theorem}
\numberwithin{equation}{section}
\theoremstyle{definition}
\newtheorem*{question}{Open Question}
\newtheorem*{remark}{Remark}
\newtheorem{exam}{Example}
\newcommand \myexam[1]{\smallskip\begin{exam}[\emph{#1}]}

\renewcommand{\phi}{\varphi}

\newcommand\lam{\lambda}

\renewcommand\hat{\widehat}

\newcommand\tld{\widetilde}
\renewcommand\bar{\overline}
\newcommand\spa{\hspace{.03in}}

\allowdisplaybreaks

\begin{document}

\title{The number of self-conjugate core partitions}

\author{Christopher R.\ H.\ Hanusa}
\address{Department of Mathematics\\
Queens College (CUNY)\\
6530 Kissena Blvd.\ \\
Flushing, NY 11367, U.S.A.\ \\
phone: +1-718-997-5964}
\email{\tt chanusa@qc.cuny.edu}

\author{Rishi Nath}
\address{Department of Mathematics and Computer Science\\
York College (CUNY)\\
Jamaica, NY 11451, U.S.A.\ \\
phone: +1-718-262-2543}
\email{\tt rnath@york.cuny.edu}

\thanks{CRH Hanusa gratefully acknowledges support from PSC-CUNY grant TRADA-42-115. Rishi Nath gratefully acknowledges support from PSC-CUNY grant TRADA-42-615}

\begin{abstract}
A conjecture on the monotonicity of $t$-core partitions in an article of Stanton [Open positivity conjectures for integer partitions, {\em Trends Math}., 2:19-25, 1999] has been the catalyst for much recent research on $t$-core partitions.  We conjecture Stanton-like monotonicity results comparing self-conjugate $(t+2)$- and $t$-core partitions of $n$.

We obtain partial results toward these conjectures for values of $t$ that are large with respect to $n$, and an application to the block theory of the symmetric and alternating groups.  To this end we prove formulas for the number of self-conjugate $t$-core partitions of $n$ as a function of the number of self-conjugate partitions of smaller $n$.  Additionally, we discuss the positivity of self-conjugate $6$-core partitions and introduce areas for future research in representation theory, asymptotic analysis, unimodality, and numerical identities and inequalities.
\end{abstract}



\thanks{Version of \today.}

\maketitle
\pagestyle{headings}

\section{Introduction}

\subsection{Background}\

In this paper we address the structure of self-conjugate core partitions.  A {\em $t$-core partition} (more briefly {\em $t$-core}) is a partition where no hook of size $t$ appears. We let $c_t(n)$ be the number of $t$-core partitions of $n$ and let $sc_t(n)$ be the number of self-conjugate $t$-core partitions of $n$. 

The study of self-conjugate partitions arises from the representation theory of the symmetric group $S_n$ and the alternating group $A_n$. At the turn of the century, Young 
discovered that the irreducible characters of $S_n$ are labeled by partitions of $n$, and in particular, the self-conjugate partitions label those that split into two conjugate irreducible representations of $A_n$ upon restriction. About the same time, Frobenius 
discovered that the hook lengths on the diagonal of a self-conjugate partition determine the irrationalities that occur in the character table of $A_n$.  

The study of core partitions also arises in representation theory; Nakayama conjectured in the forties (later proved by Brauer and Robinson) that two irreducible characters of $S_n$ are in the same $t$-block if their labeling partitions have the same $t$-core.  For this result and more on the development of the theory, see James and Kerber \cite{JamesKerber}. 
More recently, core partitions have found to be related to mock theta functions, actions of the affine symmetric group, and Ramanujan-type congruences.

Self-conjugate partitions and core partitions intersect in several important ways.  Hanusa and Jones \cite{HJ2} prove that for a fixed $t$, self-conjugate $t$-core partitions are in bijection with minimal length coset representatives in the Coxeter group quotient $\tld{C}_t/C_t$ and they determine the action of the group generators on the set of self-conjugate $t$-cores. Self-conjugate core partitions are central to an ongoing investigation into the representation-theoretic Navarro conjecture in the case of the alternating groups \cite{Nath0}.

\subsection{Positivity and monotonicity}\

The last several decades have seen a growing interest in counting core partitions; restricting to the case of self-conjugate partitions has opened new directions in research. Here we survey results on core partitions and their self-conjugate analogues and we propose a new conjecture that parallels one of Stanton. 

The $t$-core positivity conjecture asserts that every natural number has a $t$-core partition for every integer $t\geq 4$. It was finally proved by Granville and Ono \cite{GranOno} after initial results by Ono and by Erdmann and Michler.

Baldwin et al \cite{Baldwin} proved that every integer $n>2$ has a self-conjugate $t$-core partition for $t>7$, with the exception of $t=9$, for which infinitely many integers do not have such a partition. Olsson \cite{Olsson90} and Garvan, Kim, and Stanton \cite{Garvan} proved a generating function for $sc_t(n)$, succeeding Olsson's \cite{Olsson76} proof of the generating function for $c_t(n)$.  As an aside, Conjecture~\ref{conj:9core} further highlights the peculiarity of self-conjugate $9$-core partitions.

Recently, {\it simultaneous} core partitions have been investigated---partitions that are both $s$- and $t$-cores, where $s$ and $t$ are relatively prime. Anderson \cite{And02} proved that there are $\binom{s+t}{t}\big/(s+t)$ many of such partitions, and Olsson and Stanton \cite{OlssonStanton} proved that the largest such partition is of of size $n=\frac{(s^2-1)(t^2-1)}{24}$.  Ford, Mai and Sze \cite{Ford} have proved an analog of Anderson's result in the case of self-conjugate simultaneous core partitions, showing that that there are $\binom{\lfloor \frac{s}{2} \rfloor + \lfloor \frac{t}{2} \rfloor}{\lfloor \frac{t}{2} \rfloor}$ such partitions when $s$ and $t$ are relatively prime.

\medskip
In 1999, Stanton \cite{Stanton} posed the following monotonicity conjecture.

\begin{conje}[Stanton] 
\label{conj:Stanton}
Suppose that $n$ and $t$ are natural numbers and that $4\leq t\leq n-1$. Then
\[
c_{t+1}(n)\geq c_t(n).
\]
\end{conje}

This was proved for values of $t$ that are large as a function of $n$ by Craven \cite{Craven} and for large $n$ by Anderson \cite{Anderson}:

\begin{thrm}[Craven]
Suppose that $n$ is an integer, and let $t$ be an integer such that $t > 4$, and $n/2 < t < n-1$.
Then $c_t(n) < c_{t+1}(n)$.
\end{thrm}

\begin{thrm}[Anderson]
If $t_1$ and $t_2$ are fixed integers satisfying $4\leq t_1 < t_2,$ then $c_{t_1}(n)<c_{t_2}(n)$ for sufficiently large $n$.
\end{thrm}

More recently, Stanton's conjecture was proved for many more values of $t$ and $n$ by Kim and Rouse \cite{KimRouse}, including when $4\leq t\leq 198$ and $n>t+1$.

While the monotonicity criterion is conjectured for partitions in general, the set of self-conjugate partitions do not satisfy a monotonicity criterion for any $n\geq 5$.  (This is Corollary~\ref{cor:monot}; see Appendix~\ref{sec:appendix} for a table of values.)  However, we have found experimentally that $sc_{t+2}(n)\geq sc_{t}(n)$ for almost all values of $t\geq 4$ and $n\geq 4$.  Parallel to Stanton's monotonicity conjecture, we propose the following monotonicity conjectures for self-conjugate core partitions. 
\begin{conj}[Even Monotonicity Conjecture] 
\[sc_{2t+2}(n)> sc_{2t}(n) \textup{ for all } n\geq 20 \textup{ and } 6\leq 2t\leq 2\big\lfloor n/4\big\rfloor -4.\]  
\label{conj:even}
\end{conj}

\begin{conj}[Odd Monotonicity Conjecture] 
\[sc_{2t+3}(n)> sc_{2t+1}(n) \textup{ for all } n\geq 56 \textup{ and } 9\leq 2t+1\leq n-17.\]
\label{conj:odd}
\end{conj}

In this article, we discuss the given upper and lower bounds for these conjectures and prove the following partial results towards these conjectures. 

\begin{thm} 
\label{thm:evenPC}
\[sc_{2t+2}(n)> sc_{2t}(n) \textup{ when } n/4< 2t\leq 2\big\lfloor n/4\big\rfloor -4. \]  
\end{thm}

\begin{thm} 
\label{thm:oddPC}
\[sc_{2t+3}(n)> sc_{2t+1}(n) \textup{ for all } n\geq 48 \textup{ and } n/3< 2t+1\leq n-17.\]
\end{thm}

Along the way, we prove formulas for $sc_t(n)$ as a function of the number of self-conjugate partitions of $m$ for $m\leq n$ in Theorems~\ref{thm:sc2t} and \ref{thm:sc2t+1}.  As a supplement to the positivity literature, we discuss the positivity of $6$-core partitions of $n$ in Conjecture~\ref{conj:6core}.  

\subsection{Defect zero blocks of $S_n$ and $A_n$}\

For those readers familiar with the representation theory of the symmetric group $S_n$ and the alternating group $A_n$, we provide a consequence of Theorem~\ref{thm:oddPC}.  (For more information on the representation theory, see \cite[Chapter 4]{JamesKerber} or \cite[Chapter 6]{Olsson1}).

Let $t$ be an odd prime. From \cite[Proposition 12.2]{Olsson1}, we know that the defect zero $t$-blocks of $S_n$
restrict to defect zero $t$-blocks of $A_n$ in the following way.  When blocks $B_1$ and $B_2$ are labeled by distinct $t$-core partitions $\lambda_1$ and $\lambda_2$ of $n$ which satisfy $\lam_2=\lam_1^*$, then they restrict to the same defect zero $t$-block of $A_n$.  When a block $B$ is labeled by a self-conjugate partition of $n$, it splits into two distinct defect zero $t$-blocks of $A_n$ upon restriction.  These are the {\em splitting blocks} of $S_n$.

So, in particular, Theorem~\ref{thm:oddPC} implies the following.

\begin{thm}
Let $p,q$ be primes such that $p<q$ and $n/3<p,q<n-17$. For any prime $t$, let ${\mathbb B}^*_t$ be the set of defect zero $t$-blocks of $A_n$ that arise from splitting $t$-blocks of $S_n$.  Then $|{\mathbb B}^*_p|<|{\mathbb B}^*_q|$.
\end{thm}

Given a partition $\lambda$, let $\chi_{\lambda}$ be the irreducible character of $S_n$ associated to $\lam$ and consider $\prod_{i,j}h_{ij}$ the product of all the hook lengths that appear in the Young diagram of $\lambda$. The Frame--Thrall--Robinson hook length formula says that the {\it character degree} $\chi_{\lambda}(1)$ is $n!/ \prod_{i,j}h_{ij}$ \cite{Frame}. For $m\in\mathbb{Z}^+$, define $\nu_t(m)$ to be the highest power of $t$ dividing $m$.  We have the following additional corollary.

\begin{cor}
Let $p$ and $q$ be primes such that $p<q$ and $n/3<p,q<n-17$. For any prime $t$, let $Irr^*_{t}(S_n)$ be the set of irreducible characters $\chi$ of $S_n$ which split upon restriction to $A_n$ such that $\nu_{t}(|S_n|/\chi(1))=0$. Then $|Irr^*_{p}(S_n)| < |Irr^*_{q}(S_n)|$.
\end{cor}

\subsection{Organization}\

This paper is organized as follows. In Section~\ref{sec:facts}, we recall basic facts about partitions, $t$-cores, and $t$-quotients, and prove new results on self-conjugate partitions. In Section~\ref{sec:mono}, we discuss monotonicity and positivity results and conjectures depending on the parity of $t$.  Our research in self-conjugate partitions branches out in many directions---the last section of this paper brings attention to future research directions in representation theory, asymptotic analysis, unimodality, and numerical identities and inequalities.

We note that the results and perspective of Craven in \cite{Craven} motivate much of our approach, and we obtain some similar results.

\section{Self-conjugate partitions, $t$-cores and $t$-quotients} \label{sec:facts}
\subsection{Definitions}\

In order to state our results, we recall some basic definitions. More details can be found in \cite[Sections 1--2]{Olsson1} or  \cite[Chapter 2]{JamesKerber}. A {\em partition} $\lambda$ of $n$ is a non-increasing sequence $(\lambda_1,\hdots,\lambda_m)$ of positive integers such that $\sum_k \lambda_k = n$. Each $\lambda_k$ will be called a {\em component} of $\lambda$. The {\em Young diagram} associated to a partition $\lambda$ is an up- and left-aligned series of rows of boxes, where the $k$-th row has $\lambda_k$ boxes. We label the positions of boxes in the Young diagram using matrix notation; the {\em $(i,j)$-th position} is the box in the $i$-th row and $j$-th column, so that the box in position $(1,1)$ is the upper-leftmost box.  Given a partition $\lam$, its {\em conjugate} $\lambda^*$ is a partition where the number of boxes in the $k$-th column of $\lam^*$ is the number of boxes in the $k$-th row of $\lam$.  A partition is {\em self-conjugate} if $\lambda^*=\lambda$.

For a box $B$ in position $(i,j)$, its {\em hook} $H_{ij}$ is a set of boxes in the Young diagram consisting of $B$ and the set of boxes in the $i$-th row to the right of $B$ and the boxes in the $j$-th column below $B$; its {\it hook length} $h_{ij}$ is the number of boxes in $H_{ij}$.  A {\em diagonal hook} or {\em diagonal hook length} corresponds to a box on the (main) diagonal of the Young diagram.  Because a self-conjugate partition $\lambda$ is uniquely determined by its diagonal hook lengths, we will use the notation $\boldsymbol\delta(\lambda)$=($\delta_{1},\hdots,\delta_d)$ to refer to the decreasing sequence of diagonal hook lengths $h_{ii}$. If $\lambda$ contains a hook $H$ of length $k$, we say that $H$ is an $k$-hook, and we can obtain an integer partition $\lambda'$ of $n-k$ from $\lambda$ by {\em removing} $H$ in the following way: delete the boxes that constitute $H$ from the Young diagram and migrate the detached partition (if there is one) up-and-to-the-left. 

The following lemmas are related to hook lengths in self-conjugate partitions and are provided without proof.

\begin{lem} Let $\lam$ be a self-conjugate partition of $n$ defined by its diagonal hook lengths $\delta_1> \cdots> \delta_d>0$.
Then for $1\leq i\leq j\leq d$, the hook length $h_{ij}$ equals $(\delta_i+\delta_j)/2$.  When $1\leq i\leq d< j$, the hook length $h_{ij}$ is strictly less than $\delta_i/2$.
\label{p:hij}
\end{lem}

\begin{lem} Let $\lam$ be a self-conjugate partition of $n$ defined by its diagonal hook lengths $\delta_1> \cdots> \delta_d>0$.  Then $h_{ij}\leq n/2$ for all positions $(i,j)$ in the Young diagram of $\lam$, with the possible exception of $h_{11}=\delta_1$.
\label{p:bighooks}
\end{lem}

We define $SC(n)$ to be the set of self-conjugate partitions of $n$, $SC_t(n)$ to be the set of self-conjugate $t$-core partitions of $n$ and $sc(n)=|SC(n)|$ and $sc_t(n)=|SC_t(n)|$. Clearly $SC_t(n)\subseteq SC(n)$.  

The generating function for the number of $t$-core partitions is due to Olsson \cite[Proposition~3.3]{Olsson76}, while the generating function for the number of self-conjugate $t$-core partitions is due to Olsson \cite[Equation~(2.40)]{Olsson90} and Garvan, Kim, and Stanton \cite[Equation~(7.1)]{Garvan}: 
\begin{equation}
\sum_{n=0}^\infty c_t(n)q^n=\prod_{n=1}^\infty \frac{(1-q^{nt})^t}{1-q^n}
\end{equation}
\begin{equation}
\left.\sum_{n=0}^\infty sc_t(n)q^n=
\begin{cases}
\prod_{n=1}^\infty(1-q^{2tn})^{(t-1)/2}\cdot\frac{1+q^{2n-1}}{1+q^{t(2n-1)}} & \textup{if $t$ is odd}\\
\prod_{n=1}^\infty(1-q^{2tn})^{t/2}\cdot\big(1+q^{2n-1}\big) & \textup{if $t$ is even}\\
\end{cases}
\right\}.
\label{eq:scgf}
\end{equation}

The {\it $t$-core $\lambda^0$ of $\lambda$} is the partition obtained from $\lambda$ by repeatedly removing $t$-hooks until none remain; by Theorem 2.7.16 in \cite{JamesKerber}, $\lam^0$ is unique.  We introduce without definition the {\it t-quotient} of $\lambda,$ a sequence $(\lambda_{(0)},\cdots,\lambda_{(t-1)})$ of partitions which record the hooks of $\lambda$ which are divisible by $t$. We say that a $t$-quotient is {\em self-conjugate} when  $\lambda_{(k)}$ is the conjugate partition of $\lambda_{(t-1-k)}$ for all $0\leq k\leq t-1$.  The following results can be found in \cite{Olsson1} as Propositions 3.6 and 3.5.

\begin{prop} Given a partition $\lambda$ of $n$, its $t$-core $\lambda^0\vdash n_0$ and $t$-quotient\newline $(\lambda_{(0)},\cdots,\lambda_{(t-1)})$ satisfy $n=n_0 + t\sum_{k=0}^{t-1}|\lambda_{(k)}|$. Further, there are exactly $\sum_{k=0}^{t-1}|\lambda_{(k)}|$ hooks in $\lambda$ that are divisible by $t$.
\end{prop}
\begin{prop}\label{symquo} 
A partition $\lambda$ of $n$ is self-conjugate if and only if its $t$-core $\lambda^0$ and $t$-quotient $(\lambda_{(0)},\cdots,\lambda_{(t-1)})$ (with the appropriate normalization) are both self-conjugate.
\end{prop}
For the interested reader, the series of examples starting with 2.7.14 and ending with 2.7.28 in \cite{JamesKerber} provide details on how calculate the $t$-core and $t$-quotient of a partition (by way of its abacus diagram).  To show the symmetry inherent in the $t$-core and $t$-quotient of a self-conjugate partition, Figure~\ref{fig:5core} shows the $5$-core and $5$-quotient of the partition defined by diagonal hooks $\boldsymbol\delta=(29,15)$.

\begin{figure}[h]
\begin{minipage}[c]{1in}
\epsfig{figure=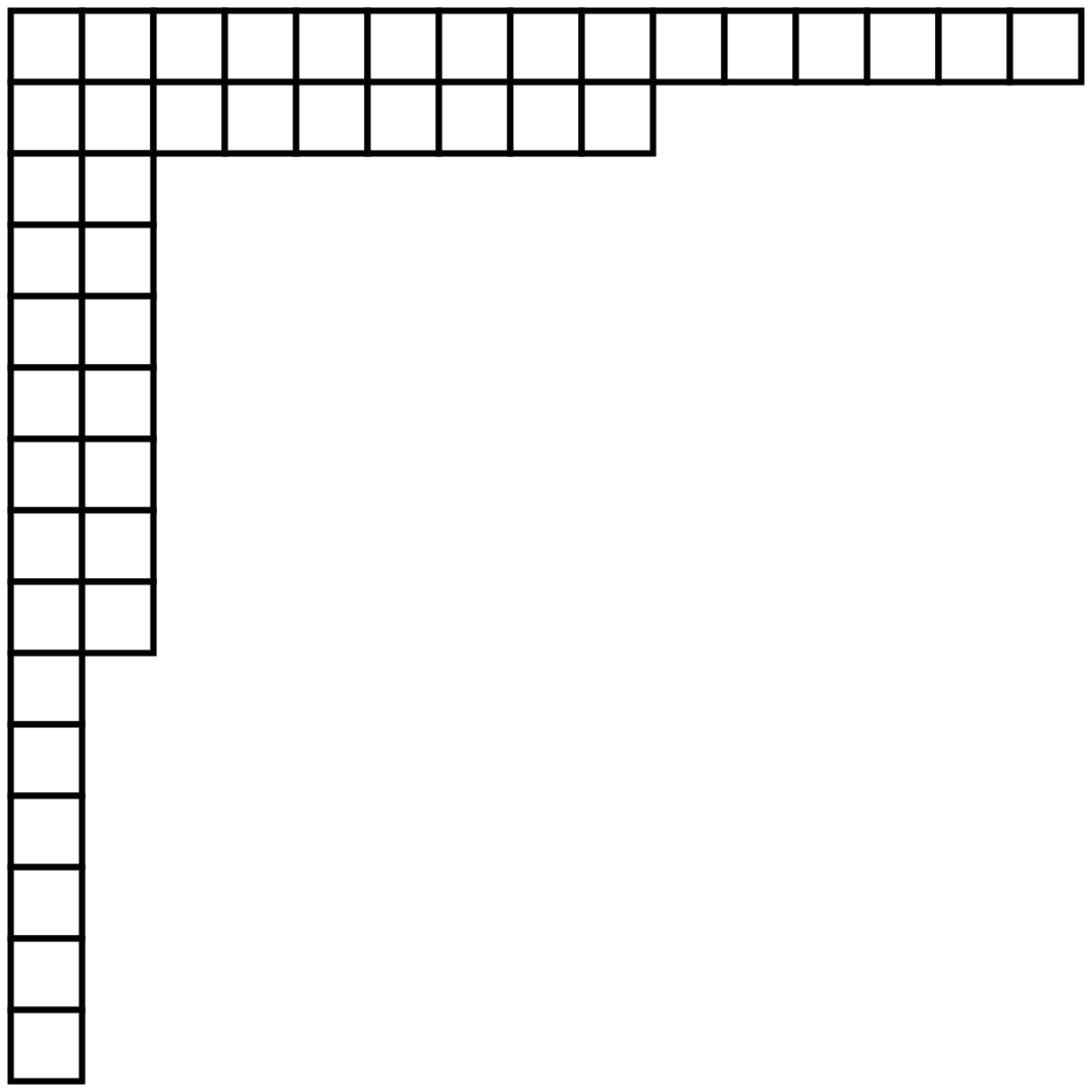,height=1in}
  \end{minipage}
 $\mapsto$\qquad
 \begin{minipage}[c]{0.75in}
 \epsfig{figure=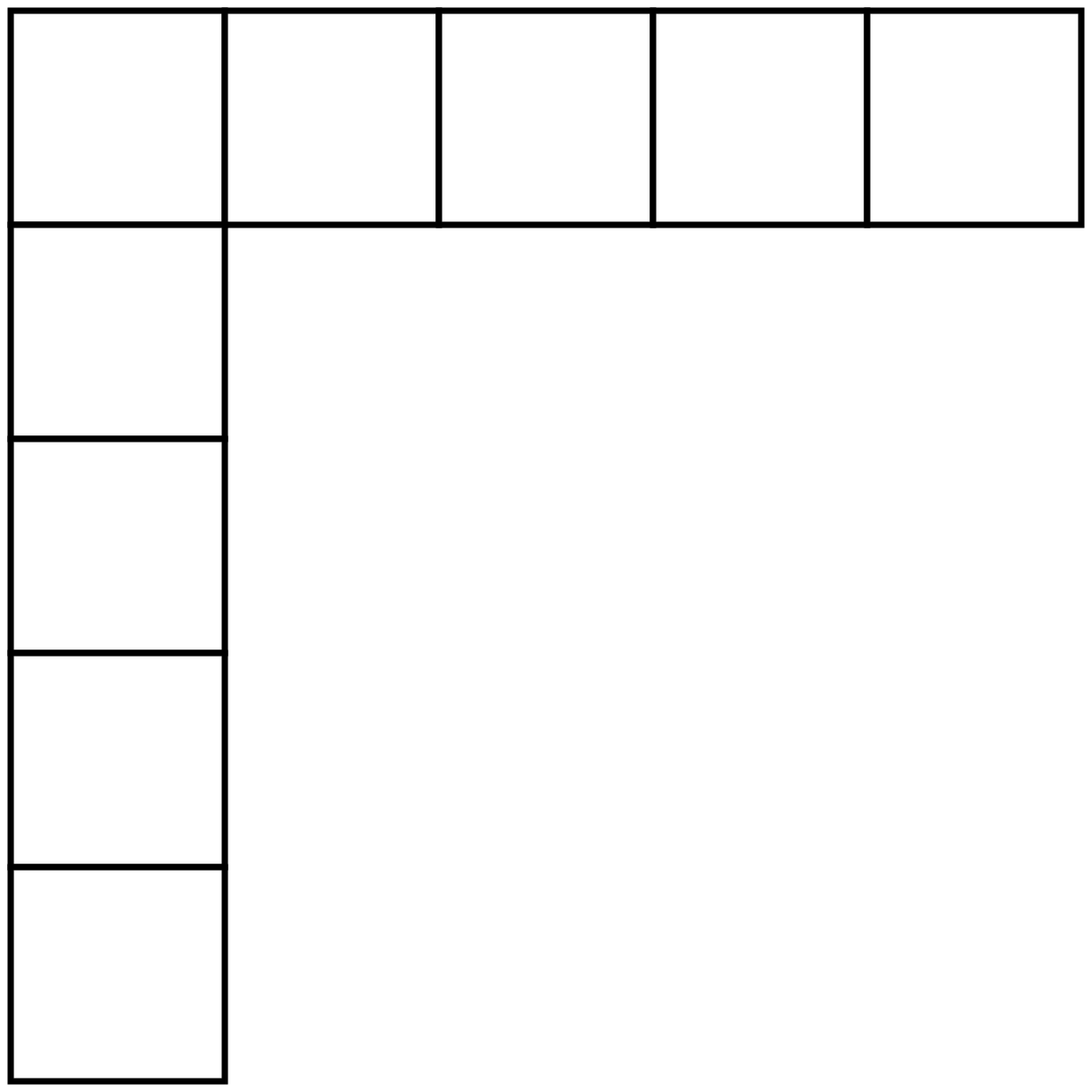,height=0.75in} 
 \end{minipage}
\quad
 and
\quad
 $\left(~
 \begin{minipage}[c]{0.2in}
 \epsfig{figure=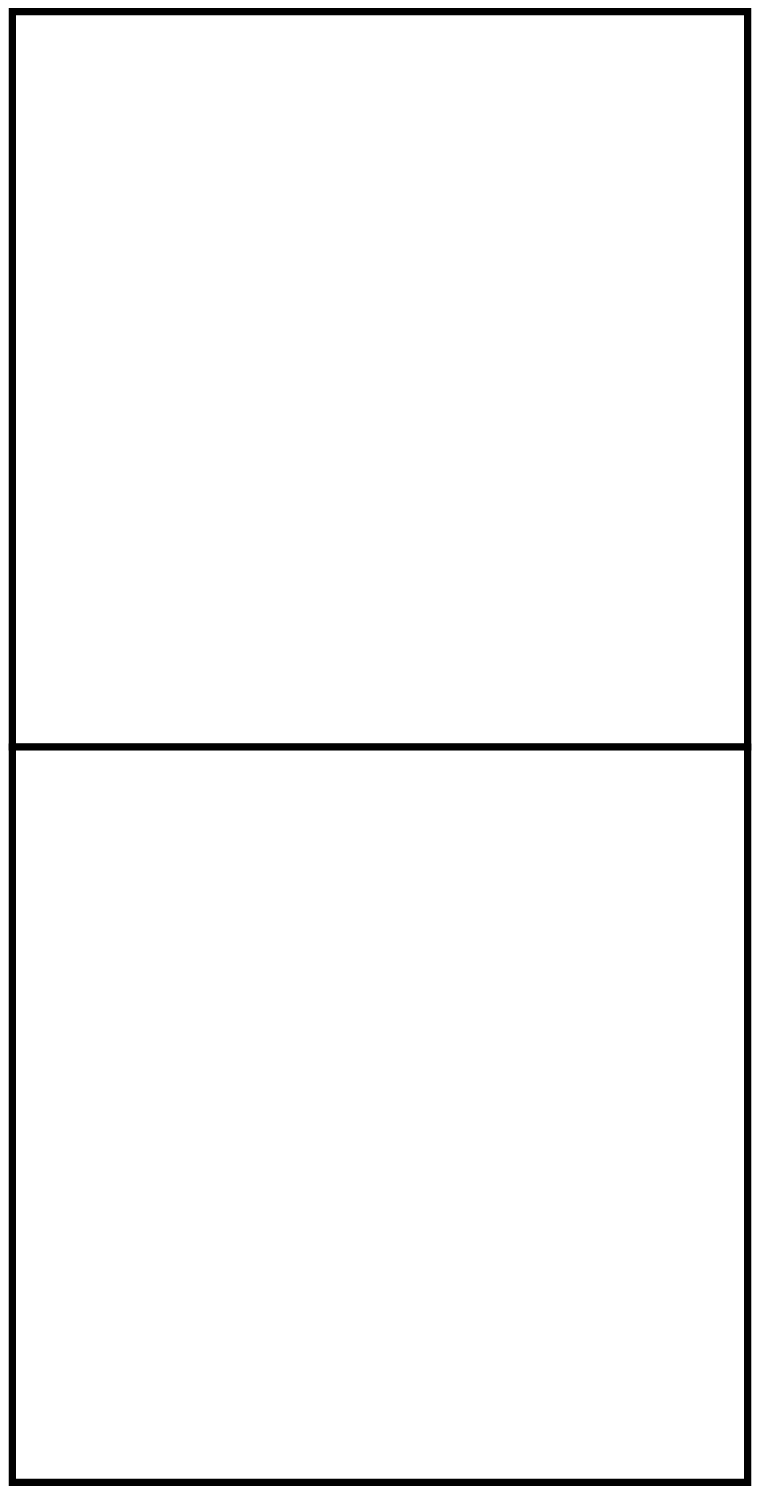,height=0.30in} 
 \end{minipage},\varnothing,
  \begin{minipage}[c]{0.35in}
 \epsfig{figure=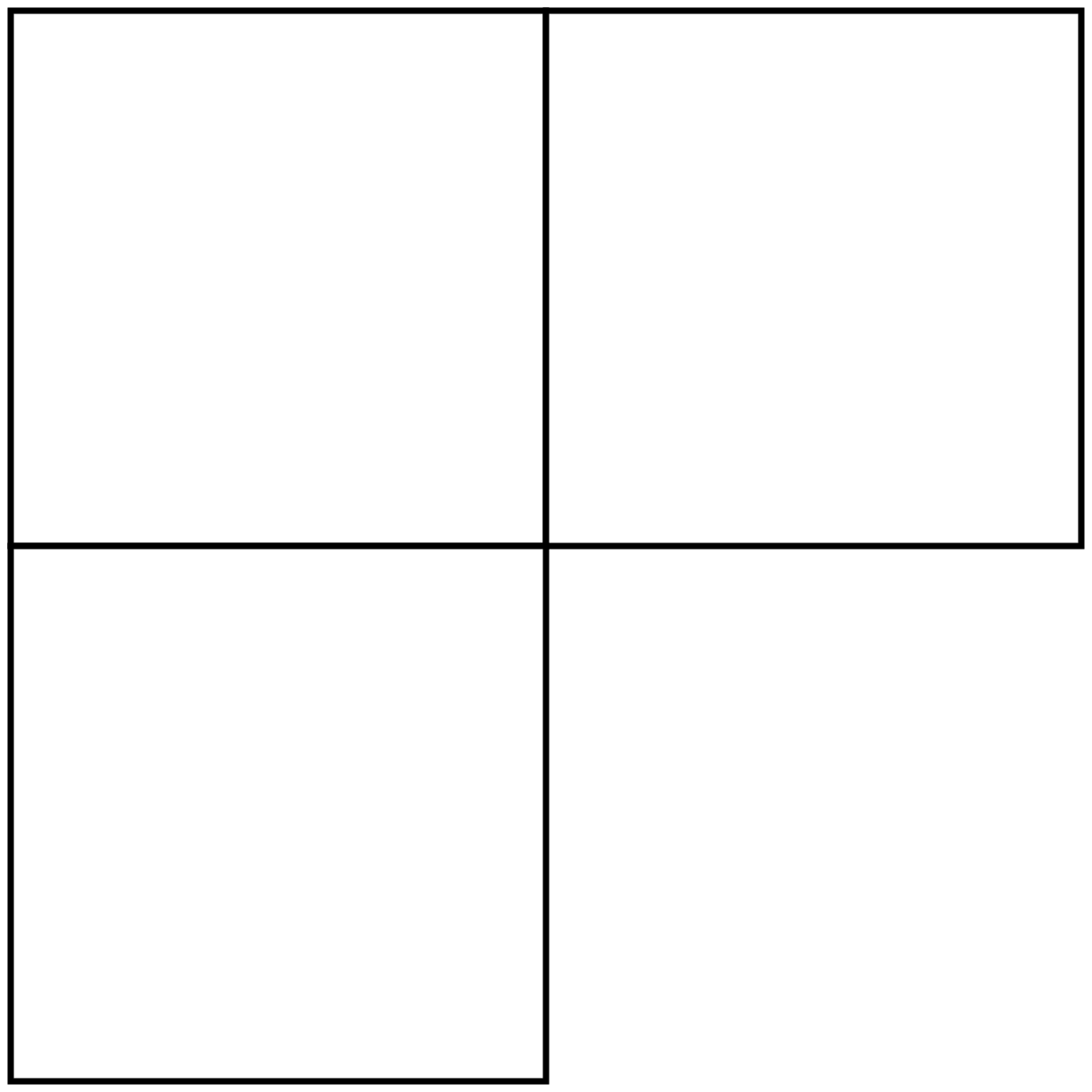,height=0.30in} 
 \end{minipage},
 \varnothing,
 \begin{minipage}[t]{0.3in}
\raisebox{0.04in}{\epsfig{figure=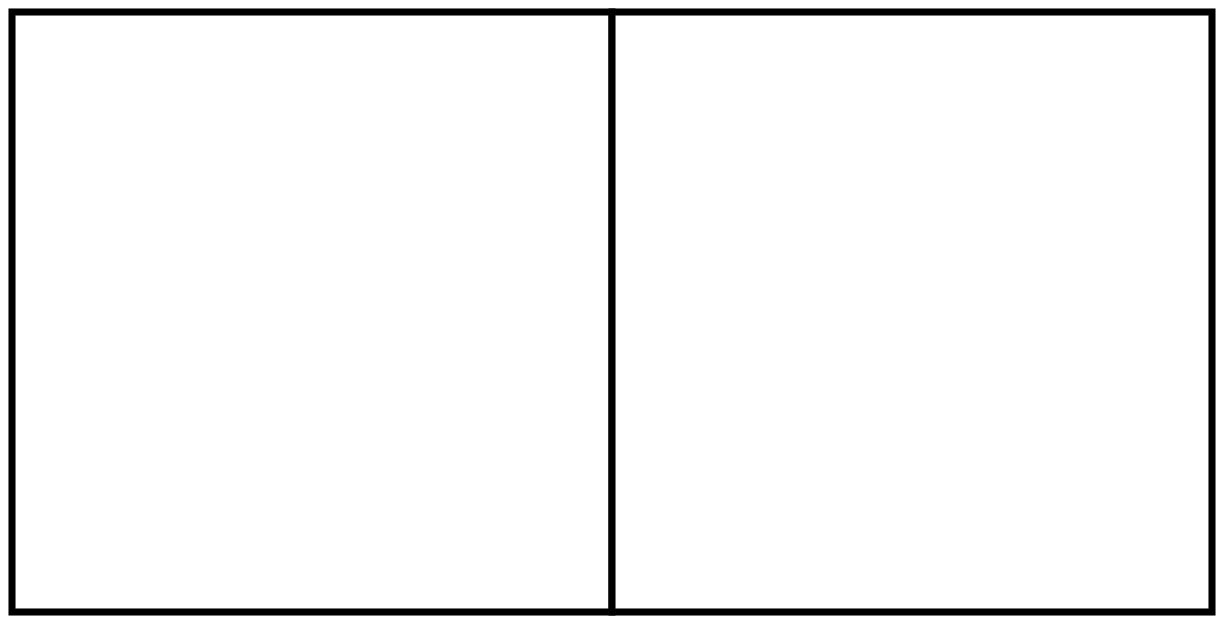,height=0.15in}}
 \end{minipage}~
 \right)$ 

\caption{The $5$-core $(5,1,1,1,1)$ and $5$-quotient $\big((1,1),\varnothing,(2,1),\varnothing,(2)\big)$ of the partition with diagonal hooks $\boldsymbol\delta=(29,15)$.}
\label{fig:5core}
\end{figure}

\subsection{Counting self-conjugate $t$-cores}\

The following result describes the possible ways to remove a minimal amount of $t$-hooks from a self-conjugate partition to obtain a self-conjugate partition.  This is discussed further in Section 4 of \cite{Nath}.

\begin{lem}\label{2tremove} 
Let $\lambda$ be a self-conjugate partition of $n$ that is not a $t$-core. 
\begin{enumerate}
\item When $t$ is even, there there exists a pair of off-diagonal $t$-hooks such that upon their removal, the resultant partition is a self-conjugate partition of $n-2t$.
\item When $t$ is odd, then one of the following must exist: a pair of off-diagonal $t$-hooks as in \textup{(1)} or a diagonal $t$-hook such that upon its removal, the resultant partition is a self-conjugate partition of $n-t$.
\end{enumerate}
\end{lem}

The following result is key in proving our main results.
\begin{thm}
Let $n$ and $t$ be positive integers.  Then  
\begin{equation}
sc_{2t}(n)=sc(n)-\sum_{1\leq i\leq \lfloor \frac{n}{4t}\rfloor} sc_{2t}(n-4it)\,\hat{p}_{t}(i)
\label{eq:EvenSum}
\end{equation}
and
\begin{equation}
sc_{2t+1}(n)=sc(n)-\hspace{-.1in}\sum_{\substack{i,j\geq 0\\ 1 \leq 2i+j\leq \lfloor \frac{n}{2t+1} \rfloor}} \hspace{-.1in}sc_{2t+1}\big(n-(2i+j)(2t+1)\big)\,\hat{p}_{t}(i)\,sc(j),
\label{eq:OddSum}
\end{equation}
where $\hat{p}_t(n)$ is the number of sequences of length $t$ of (possibly empty) partitions $\lambda_{(k)}$ such that $\sum_{k}|\lambda_{(k)}|=n$. 
\end{thm}
\begin{proof} 
Consider the set $\bar{SC}_{2t}(n)$ of self-conjugate partitions of $n$ that are not $2t$-cores and let $\bar{sc}_{2t}=|\bar{SC}_{2t}(n)|$, whereby $sc(n)=sc_{2t}(n)+\bar{sc}_{2t}(n)$.  By Lemma~\ref{2tremove}, the $2t$-core of any non-$2t$-core must be obtained by the removing an even number of $2t$-hooks. Furthermore, by Proposition~\ref{symquo}, its $2t$-core and (non-empty) $2t$-quotient are both self-conjugate. When one removes $2i$ $2t$-hooks, the $2t$-core is a partition of $n-(2i)(2t)$ and there are $\hat{p}_{t}(i)$ possible $2t$-quotients. Summing over valid values of $i$ gives Equation~\eqref{eq:EvenSum}.

\medskip
Consider the set $\bar{SC}_{2t+1}(n)$ of self-conjugate partitions of $n$ that are not $(2t+1)$-cores and let $\bar{sc}_{2t+1}(n)=|\bar{SC}_{2t+1}(n)|$. The argument proceeds similarly as above, with the additional condition that the core of a non-$(2t+1)$-core can be obtained by removing $2i$ off-diagonal $(2t+1)$-hooks and/or $j$ diagonal $(2t+1)$-hooks, in which case the $(2t+1)$-quotient has a non-empty partition $\lam_{(t+1)}$ of $j$ that is itself self-conjugate. (Note that this means $j$ will never be $2$.)  There are a total of $\hat{p}_{t}(i)\,sc(j)$ possible $(2t+1)$-quotients which remove a total of $(2i+j)$ $(2t+1)$-hooks, and their $(2t+1)$-cores are partitions of $n-(2i+j)(2t+1)$.  Summing over valid values of $i$ and $j$ gives Equation~\eqref{eq:OddSum}.
\end{proof}

\subsection{Bounding the growth of $sc(n)$}\

We establish bounds on $\frac{sc(n-2)}{sc(n)}$ and $\frac{sc(n-4)}{sc(n)}$, which will be used in the next section to prove Theorems~\ref{thm:evenPC} and \ref{thm:oddPC}.  The technique used here is an adaptation of Section 3 in \cite{Craven}. 

\begin{lem}
Let $n$ be an integer greater than or equal to $19$.  Then $\frac{sc(n-2)}{sc(n)}< \frac{n}{n+2}$.
\label{lem:scn-2}
\end{lem}
\begin{proof}
For a given $n\geq 27$, define two sets of self-conjugate partitions:
\begin{itemize}
\item[$A_n$:] The set of self-conjugate partitions of $n$ whose diagonal hooks satisfy $\delta_1-\delta_2\geq 4$.  If $n$ is odd, also include $\boldsymbol\delta=(n)$.
\item[$B_n$:] The set of self-conjugate partitions of $n$ whose diagonal hooks satisfy $\delta_1=\delta_2+2$ and whose parts are not all the same (when $n$ is a square number).
\item[$C_n$:] The set of self-conjugate partitions of $n$ in neither $A_n$ nor $B_n$.
\end{itemize}

There is a bijection $f:SC(n-2) \rightarrow A_n$  which takes a self-conjugate partition of $n-2$ and adds one box to the first row and to the first column.  We conclude that $|A_n|=sc(n-2)$.

When $B_n$ is nonempty, there is also an surjection $g:SC(n-2) \twoheadrightarrow B_n$.  ($B_n$ is nonempty for all values of $n\geq 19$.)  For $\lam\in SC(n-2)$, define $g(\lam)\in B_n$ by the following steps.  First, if $\lam$ has one diagonal hook, define $g(\lam)$ to have diagonal hooks $(\frac{n+1}{2},\frac{n-3}{2},1)$ if $n\equiv 1\mod 4$ or $(\frac{n-1}{2},\frac{n-5}{2},3)$ if $n\equiv 3\mod 4$. Otherwise, suppose that the diagonal hooks of $\lam$ are $\boldsymbol\delta=(\delta_1,\hdots,\delta_d)$; create a self-conjugate partition $\lam'$ with diagonal hooks $\boldsymbol\delta'=(\delta_1',\hdots,\delta_d')$, where
\begin{equation*}
\left.
\begin{cases}
\delta_1'=\frac{\delta_1+\delta_2}{2}+2 \textup{ and } \delta_2'=\frac{\delta_1+\delta_2}{2}-2 & \textup{if $\frac{\delta_1+\delta_2}{2}$ is odd} \\
\delta_1'=\frac{\delta_1+\delta_2}{2}+1 \textup{ and } \delta_2'=\frac{\delta_1+\delta_2}{2}-1 & \textup{if $\frac{\delta_1+\delta_2}{2}$ is even}
\end{cases}
\right\},
\end{equation*}
and which keeps all other diagonal hooks the same ($\delta_i'=\delta_i$ for all $3\leq i\leq d$).  Next, determine (if it exists) the first $i$ such that $\delta_i'\geq \delta_{i+1}'+4$.  Define $g(\lam)$ to be the partition which adds one box to the $(i+1)$-st row and to the $(i+1)$-st column of $\lam'$.  If no such $i$ exists, then $\boldsymbol\delta'$ is of the form $(2m+1,2m-1,\hdots,2k+3,2k+1)$ for $m> k>0$.  If $\lam$ has two diagonal hooks, then define $g(\lam)$ to have diagonal hooks $(\frac{n-4}{2},\frac{n-8}{2},5,1)$.  
Otherwise, $\lam'$ has three or more diagonal hooks and $\delta_d'>1$; define $g(\lam)$ to have diagonal hooks $(\delta_1'+2,\delta_2'+2,\delta_3',\hdots,\delta_d'-2)$. 

The function $g$ is well defined because the image of every self-conjugate partition satisfies $\delta_1=\delta_2+2$  and is a surjection because for the function $h:B_n\rightarrow sc(n-2)$ that removes the last box in the last row and the last box in the last column, then for any $\beta\in B_n$, it is true that $g(h(\beta))=\beta$.  

For $\beta\in B_n$, define the set $\Lambda_\beta\subset SC(n-2)$ to be the preimages of $\beta\in B_n$ under $g$.  The largest that this set can be is for the following $\beta^*\in B_n$, with diagonal hooks 
\[\left.\boldsymbol\delta(\beta^*)=
\begin{cases}
\big((n+2)/2,(n-2)/2\big) & n\equiv 0\mod 4 \\
\big((n+1)/2,(n-3)/2,1\big) & n\equiv 1\mod 4 \\
\big((n-4)/2,(n-8)/2,5,1\big) & n\equiv 2\mod 4 \\
\big((n-1)/2,(n-5)/2,3\big) & n\equiv 3\mod 4 \\
\end{cases}
\right\}.\]
In each of these cases, $|\Lambda_{\beta^*}|<n/2$.

From the definitions of $f$ and $g$, we can now bound $sc(n)$ as a function of $sc(n-2)$ when  $n\geq 27$:
\[sc(n)=|A_n|+|B(n)|+|C(n)|> sc(n-2)+sc(n-2)/(n/2)+0=\frac{n+2}{n}sc(n-2).\]
The equation $\frac{sc(n-2)}{sc(n)}< \frac{n}{n+2}$ also holds for $19\leq n \leq 26$.
\end{proof}

\begin{lem}
Let $n$ be an integer greater than or equal to $8$.  Then $\frac{sc(n-4)}{sc(n)}< \frac{n}{n+4}$.
\label{lem:scn-4}
\end{lem}
\begin{proof}Lemma~\ref{lem:scn-2} implies \[\frac{sc(n-4)}{sc(n)}=\frac{sc(n-4)}{sc(n-2)}\cdot\frac{sc(n-2)}{sc(n)}< \frac{n-2}{n}\cdot\frac{n}{n+2}=\frac{n-2}{n+2}<\frac{n}{n+4}\]
for $n\geq 21$.  The equation $\frac{sc(n-4)}{sc(n)}< \frac{n}{n+4}$ also holds for $8\leq n \leq 20$.
\end{proof}

\begin{remark}
The sequence $\{sc(n)\}_{n\geq 0}$ (A000700 in the On-Line Encyclopedia of Integer Sequences \cite{oeis}) starts 
\[\{1, 1, 0, 1, 1, 1, 1, 1, 2, 2, 2, 2, 3, 3, 3, 4, 5, 5, 5, 6, 7, 8, 8, 9, 11, 12, 12, 14\}.\] This, and Lemma~\ref{lem:scn-2} implies that $sc(n+2)>sc(n)$ for integers $n\geq 17$.  It also follows that $sc(n+2)-sc(n)>1$ for $n\geq 24$.  
\end{remark}

\section{Main Results}
\label{sec:mono}

In this section, we prove formulas for $sc_{t}(n)$ for certain values of $t$ and $n$, discussing their consequences for our monotonicity conjectures and the positivity of self-conjugate $t$-cores.

\subsection{Monotonicity in large $2t$-cores}\

We first discuss formulas for $sc_{2t}(n)$ for large values of $2t$.

\medskip
Because the largest diagonal hook $\delta_1$ is odd in every self-conjugate core partition, we have the following corollary of Lemma~\ref{p:bighooks}.
\begin{cor} 
\label{cor:n/2}
Every self-conjugate partition of $n$ is a $2t$-core for all integers $t$ satisfying $2t>n/2$. In particular,  $sc_{2t}(n)=sc(n)$ for integers $t$ satisfying $2t>n/2$. 
\end{cor}

Proposition~\ref{prop:evenLarge} establishes a simple formula for $sc_{2t}(n)$ for values of $2t$ between $n/4$ and $n/2$, which will be useful for proving Theorem~\ref{thm:evenPC}. 

\begin{prop}
\label{prop:evenLarge}
Let $n$ be a positive integer and suppose $t$ is an integer satisfying $n/4<2t\leq n/2$.  Then 
\begin{equation} 
\label{eq:evenLarge}
sc_{2t}(n)= sc(n)-t\,sc(n-4t).
\end{equation}
\end{prop}
\begin{proof}
When $n/4<2t\leq n/2$, the sum in Equation~\eqref{eq:EvenSum} consists only of its first term, $sc_{2t}(n-4t)\hat{p}_t(1)$.  Equation~\eqref{eq:evenLarge} follows because $\hat{p}_t(1)=t$ and from Corollary~\ref{cor:n/2} because $2t>(n-4t)/2$.
\end{proof}

We must be careful for values of $2t$ near $n/2$.  Substituting $2t=2\lfloor n/4\rfloor$ and $2t=2\lfloor n/4\rfloor-2$ into Equation~\eqref{eq:evenLarge} establishes that when $n\geq 12$ and $n\not\equiv 2\mod 4$, $sc_{2\lfloor n/4\rfloor}(n)=sc_{2\lfloor n/4\rfloor-2}(n)-1$, which explains the upper bound we give for the even monotonicity conjecture.  Explicit formulas are given as Corollary~\ref{cor:2n/4}. 

\begin{cor}
\label{cor:2n/4}
 Let $n$ be an integer greater than or equal to $4$. Then
\begin{equation*} 
sc_{2\lfloor n/4\rfloor}(n)=\left.\begin{cases} 
sc(n)-\lfloor n/4\rfloor& \textup{when $n\equiv 0,1,3\mod 4$} \\ 
sc(n)& \textup{when $n\equiv 2\mod 4$}\end{cases}
\right\}.
\end{equation*}
Furthermore, let $n$ be an integer greater than or equal to $12$. Then
\begin{equation*}
sc_{2\lfloor n/4\rfloor-2}(n)=sc(n)-(\lfloor n/4\rfloor-1).
\end{equation*}
\end{cor}
\begin{proof}
In the formula for $sc_{2t}(n)$, the coefficients of $t$ are simply $sc(n-4t)$, which depends on $n$ modulo $4$.  The range for which the formulas are valid comes from solving $n/4\leq 2\lfloor n/4\rfloor$ or $n/4\leq 2\lfloor n/4\rfloor-2$.  
\end{proof}

\begin{remark}
For successively smaller values of $2t$, formulas similar to those in Corollary~\ref{cor:2n/4} can be found.  For example, when $n$ is an integer greater than or equal to $52$, then
\begin{equation*} 
sc_{2\lfloor n/4\rfloor-12}(n)=\left.\begin{cases} 
sc(n)-11(\lfloor n/4\rfloor-6)& \textup{when $n\equiv 0\mod 4$} \\ 
sc(n)-12(\lfloor n/4\rfloor-6)& \textup{when $n\equiv 1,2\mod 4$} \\ 
sc(n)-14(\lfloor n/4\rfloor-6)& \textup{when $n\equiv 3\mod 4$}\end{cases}
\right\}.
\end{equation*}
In general for self-conjugate $(2\lfloor n/4\rfloor-2i)$-cores, the range of validity of the formula for $sc_{2\lfloor n/4\rfloor-2i}$ is for $n\geq 4(2i+1)$.  While Equation~\eqref{eq:evenLarge} does encompass all formulas of this type, these formulas are interesting in their own right.
\end{remark}

In general, we can apply Equation~\eqref{eq:EvenSum} repeatedly to find a formula for $sc_{2t}(n)$ for all values of $2t$; the formula only involves polynomials of $t$ and values of $sc(m)$ for $m\leq n$.
\begin{thm}
\label{thm:sc2t}
We have the following formula for $sc_{2t}(n)$.
\[sc_{2t}(n)=\sum_{\substack{I=(i_1,\hdots,i_k)\\|I|\leq\lfloor \frac{n}{4t}\rfloor}} (-1)^k \hat{p}_t(i_1)\cdots\hat{p}_t(i_k)\,sc(n-4|I|t),\]
where the sum is over all sequences of positive integers $I=(i_1,\hdots,i_k)$ such that its sum $|I|=i_1+\cdots+i_k\leq\lfloor \frac{n}{4t}\rfloor$ 
\end{thm}
We now prove Theorem \ref{thm:evenPC}.
\begin{proof}[Proof of Theorem~\ref{thm:evenPC}]
By Equation~\eqref{eq:evenLarge}, it suffices to prove $(t+1)\,sc(n-4t-4)<t\,sc(n-4t)$, which is equivalent to $\frac{sc(n-4t-4)}{sc(n-4t)}<\frac{t}{t+1}$.  Lemma~\ref{lem:scn-4} implies that $\frac{sc(n-4t-4)}{sc(n-4t)}< \frac{n-4t}{n-4t+4}$ when $n-4t\geq 8$; this condition is satisfied because the upper bound for $2t$ under consideration implies $n-4t\geq n-4(\lfloor n/4\rfloor-2)\geq 8$.  

Last, because $n/4< 2t$ then $(n-4t)(t+1)< t(n-4t+4)$, from which we have $\frac{sc(n-4t-4)}{sc(n-4t)}\leq\frac{n-4t}{n-4t+4}< \frac{t}{t+1}$.  This completes the proof.
\end{proof}

\subsection{Positivity and monotonicity in small $2t$-cores}\

Before discussing monotonicity for small values of $2t$, we first discuss what is known about positivity in self-conjugate $2t$-core partitions.

The only partitions which are $2$-cores are the staircase partitions $\lam=(k,k-1,\hdots,2,1)$, which are all self-conjugate.  As a consequence, $sc_2(n)$ is non-zero exactly when $n$ is a triangular number.  Ono and Sze \cite[Theorem 3]{OnoSze} characterize the integers having no self-conjugate $4$-core: $sc_4(n)=0$ if and only if the prime factorization of $8n+5$ contains a prime of the form $4k+3$ to an odd power.

Baldwin et al.\ \cite{Baldwin} prove that $sc_t(n)$ is positive for $t\geq 8$ and $n\ne 2$, and give the example of $sc_6(13)=0$ to show that $sc_6(n)$ is not always positive. However, they do not characterize when $sc_6(n)$ is zero.  By using its generating function, we generated the values of $sc_6(n)$ for $0\leq n\leq 10000$, from which we conjecture the following.

\begin{conj} Let $n$ be a positive integer.  Then $sc_6(n)> 0$ except when $n\in\{2,12,13,73\}$.
\label{conj:6core}
\end{conj}

In the even monotonicity conjecture, we give the lower bound $2t$ equals $6$.  Indeed, there are integers $n$ such that  $sc_6(n)\leq sc_4(n)$, even for values of $n$ larger than $15$. (Corollary~\ref{cor:2n/4}  establishes that $sc_6(15)<sc_4(15)$.)  We conjecture that the set of such integers is finite, again aided by a computer search of non-negative integers $n$ up to $10000$.

\begin{conj} Let $n$ be an integer larger than $15$. Then $sc_6(n)<sc_4(n)$ when $n\in\{112,180,265\}$ and $sc_6(n)=sc_4(n)$ when $n\in\{27, 28, 33, 40, 73, 75, 118, 190, 248\}$.
\end{conj}

There are no values of $20 \leq n \leq 10000$ such that $sc_8(n)\leq sc_6(n)$.

\subsection{Monotonicity in large $(2t+1)$-cores}\

For $2t+1>n$, there are no partitions of $n$ containing a hook length of $2t+1$.  By Lemma~\ref{p:bighooks}, we know that the values of $sc_{2t+1}(n)$ for $2t+1>n/2$ are determined by the number of self-conjugate core partitions that have $2t+1$ as its first diagonal hook.  In other words, 

\begin{cor}
Let $n$ be a positive integer and suppose that $t$ satisfies $n/2<2t+1\leq n$.  Then \[sc_{2t+1}(n)=sc(n)-sc(n-2t-1).\]
\label{cor:n/2odd}
\end{cor}

Corollaries~\ref{cor:n/2} and \ref{cor:n/2odd} imply:

\begin{cor}
For fixed $n\geq 5$, the sequence $\{sc_t(n)\}_{t\geq 2}$ is not monotonic.
\label{cor:monot}
\end{cor}

Corollary~\ref{cor:n/2odd} also implies that for $t$ satisfying $n/2<2t+1\leq n-2$,
\(sc_{2t+3}(n)-sc_{2t+1}(n)=sc(n-2t-3)-sc(n-2t-1)\).  Because $sc(n+2)>sc(n)$ for integers $n\geq 17$, we have the following corollary.

\begin{cor}
Let $n$ be a positive integer and suppose that $t$ satisfies $n/2<2t+1\leq n-17$.  Then \
\[sc_{2t+3}(n)>sc_{2t+1}(n).\]
\label{cor:largeOdd}
\end{cor}

We now establish a formula for $sc_{2t+1}(n)$ for values of $2t+1$ between $n/3$ and $n/2$. 

\begin{prop}
\label{prop:oddLarge}
Let $n$ be a positive integer and suppose $t$ is an integer satisfying $n/3<2t+1\leq n/2$.  Then 
\begin{equation} 
\label{eq:oddLarge}
sc_{2t+1}(n)= sc(n)-sc(n-2t-1)-(t-1)\,sc(n-4t-2).
\end{equation}
\end{prop}
\begin{proof}
When $n/3<2t+1\leq n/2$, the sum in Equation~\eqref{eq:OddSum} is the sum of only two non-zero terms,
\[sc_{2t+1}(n-2t-1)\,\hat{p}_t(0)\,sc(1)+t\,sc_{2t+1}(n-4t-2)\,\hat{p}_t(1)\,sc(0),\]
which simplifies to $sc_{2t+1}(n-2t-1)+t\,sc_{2t+1}(n-4t-2)$. We remark that $2t+1>(n-4t-2)/2$ and $(n-2t-1)/2<2t+1\leq (n-2t-1)$, so Corollary~\ref{cor:n/2odd} implies that 
\[|\bar{SC}_{2t+1}(n)|=\big[sc(n-2t-1)-sc(n-4t-2)\big]+t\,sc(n-4t-2),\]
from which Equation~\eqref{eq:oddLarge} follows.
\end{proof}

\begin{remark}
Formulas like those in Corollary~\ref{cor:2n/4} can be found now for odd cores.  For example, when $n$ be an integer greater than or equal to $76$, then
\begin{equation*} 
sc_{2\lfloor n/4\rfloor-11}(n)=\left.\begin{cases} 
sc(n)-sc(n-2\lfloor n/4\rfloor+11)-8(\lfloor n/4\rfloor-7)& \textup{when $n\equiv 0\mod 4$} \\ 
sc(n)-sc(n-2\lfloor n/4\rfloor+11)-9(\lfloor n/4\rfloor-7)& \textup{when $n\equiv 1\mod 4$} \\ 
sc(n)-sc(n-2\lfloor n/4\rfloor+11)-11(\lfloor n/4\rfloor-7)& \textup{when $n\equiv 2\mod 4$} \\
sc(n)-sc(n-2\lfloor n/4\rfloor+11)-12(\lfloor n/4\rfloor-7)& \textup{when $n\equiv 3\mod 4$}\end{cases}
\right\}.
\end{equation*}
\end{remark}

As in the even core case, we can use Equation~\eqref{eq:OddSum} to find a formula for $sc_{2t+1}(n)$ involving polynomials of $t$ and values of $sc(m)$ for $m\leq n$.
\begin{thm}
\label{thm:sc2t+1}
We have the following formula for $sc_{2t+1}(n)$.
\[sc_{2t+1}(n)=\sum_{\substack{I=(i_1,\hdots,i_k)\\J=(j_1,\hdots,j_k)\\2|I|+|J|\leq\lfloor \frac{n}{2t}\rfloor}} (-1)^k \hat{p}_t(i_1)\cdots\hat{p}_t(i_k)\,sc(j_1)\cdots sc(j_k)\,sc\big(n-(2|I|+|J|)(2t+1)\big),\]
where the sum is over all pairs of sequences of non-negative integers $I=(i_1,\hdots,i_k)$ and $J=(j_1,\hdots,j_k)$ such that $i_l+j_l\geq 1$ for all $1\leq l\leq k$ and their sums satisfy $2|I|+|J|\leq\lfloor\frac{n}{2t+1}\rfloor$.
\end{thm}

We now prove Theorem~\ref{thm:oddPC}.
\begin{proof}[Proof of Theorem~\ref{thm:oddPC}]
Alongside Corollary~\ref{cor:largeOdd} it remains to establish that $sc_{2t+3}(n)>sc_{2t+1}(n)$ for $n/3<2t+1\leq n/2$. 

When $n/2-2<2t+1\leq n/2$ and $n\geq 34$, then $n-2t-3>n/2\geq 17$, so $sc(n-2t-1)>sc(n-2t-3)$ and we have
\[sc_{2t+3}(n)=sc(n)-sc(n-2t-3)>sc(n)-sc(n-2t-1)-(t-1)\,sc(n-4t-2)=sc_{2t+1}(n).\]

When $2t+1\leq n/2-2$, Proposition \ref{prop:oddLarge} implies we need to prove
\begin{equation}
t\,sc(n-4t-6)+sc(n-2t-3)<(t-1)\,sc(n-4t-2)+sc(n-2t-1).
\label{eq:simpleSC}
\end{equation}

When $n/2-4<2t+1\leq n/2-2$, then $4\leq n-4t-2\leq 7$ and $0\leq n-4t-6\leq 3$.  If $n-4t-2=6$, then Equation~\eqref{eq:simpleSC} is $sc(n-2t-3)<(t-1)+sc(n-2t-1)$, which is certainly true when $sc(n-2t-1)>sc(n-2t-3)$.  Otherwise, $sc(n-4t-2)=sc(n-4t-6)=1$, so Equation~\eqref{eq:simpleSC} becomes $1+sc(n-2t-3)<sc(n-2t-1)$, which is true when $n-2t-1>26$; for the given range of $2t+1$, this requires $n>48$.  (This result also holds for $n=48$ and $t=21$.)

When $n/3<2t+1\leq n/2-4$, we will prove Equation~\eqref{eq:simpleSC} by proving that $t\,sc(n-4t-6)<(t-1)\,sc(n-4t-2)$ and relying on the fact that $sc(n-2t-1)\geq sc(n-2t-3)$ for $n$ and $t$ in our range. 
Since $n/2-4\geq 2t+1$, then $n-4t-2\geq 8$, so Lemma~\ref{lem:scn-4} applies to give $\frac{sc(n-4t-6)}{sc(n-4t-2)}< \frac{n-4t-2}{n-4t+2}$.  When $n>18$, then $(n+6)/4<n/3$, which in turn is less than $2t+1$.  Therefore $n+2<8t$, so $(n-4t-2)t<(n-4t+2)(t+1)$, implying $\frac{sc(n-4t-6)}{sc(n-4t-2)}\leq \frac{n-4t-2}{n-4t+2}<\frac{t-1}{t}$, from which $t\,sc(n-4t-6)<(t-1)\,sc(n-4t-2)$, as desired.  
\end{proof}

\begin{remark}
The lower bound of $n=48$ in Theorem~\ref{thm:oddPC} is necessary---from Proposition~\ref{prop:oddLarge}, we have $sc_{23}(47)=sc_{21}(47)$, $sc_{21}(45)=sc_{19}(45)$, $sc_{21}(42)=sc_{19}(42)$, $sc_{19}(39)=sc_{17}(39)$,  and $sc_{17}(37)=sc_{15}(37)$.  There are other anomalies in for other values of $n\leq 41$ and $t\geq 11$: we have $sc_{13}(34)=sc_{11}(34)$, $sc_{15}(39)=sc_{13}(39)$, $sc_{13}(41)=sc_{11}(41)$. Also of note are the two cases $sc_{13}(29)<sc_{11}(29)$ and $sc_{15}(31)<sc_{13}(31)$.
\end{remark}

\subsection{Positivity and monotonicity in small $(2t+1)$-cores}\

Robbins \cite[Theorem 7]{Robbins} and Baruah and Berndt \cite[Theorem 5.2]{BB} prove that the only integers having at least one self-conjugate $3$-core (in fact, there is exactly one) are of the form $3d^2+2d$ or $3d^2-2d$ for some non-negative integer $d$. 

Garvan, Kim, and Stanton \cite{Garvan} characterize the integers having no self-conjugate $5$-core: $sc_5(n)=0$ if and only if the prime factorization of $n$ contains a prime of the form $4k+3$ to an odd power.  In addition, they cite an observation of Doug McDoniel involving representations of integers as sums of three squares that proves that $sc_7(n)=0$ if and only if $n=(8m+1)4^k-2$ for integers $m$ and $k$.

Baldwin et al.\ \cite{Baldwin} prove that $sc_9(n)=0$ for all $n$ of the form $n=(4^k-10)/3$ and cite a communication with Peter Montgomery which proves that this is a complete characterization of integers having no self-conjugate $9$-core partitions.

In the odd monotonicity conjecture, we give the lower bound $2t+1$ equals $9$.  Unlike in the even case, it appears that $sc_9(n)<sc_7(n)$ for infinitely many values of $n$; the non-negative values of $n$ up to $10000$ for which $sc_9(n)<sc_7(n)$ are  
\begin{align*}
\{ &9, 18, 21, 82, 114, 146, 178, 210, 338, 402, 466, 594, 658, 722, 786, 850, 978, \\ & 1106, 1362, 1426,  1618, 1746, 1874, 2130, 2386, 2514, 2642, 2770, 2898,  3154, 3282, \\ & 3410, 3666, 3922, 4050,  4178, 4306, 4434, 4690, 4818, 4946, 5202, 5458, 5586, 5970, \\ & 6226, 6482, 6738, 6994, 7250, 7506, 8018, 8274, 8530, 8786, 9042, 9298, 9554, 9810\}.
\end{align*}
Note that these include many (but not all) values of $n\equiv 82\mod 128$; this condition is neither necessary nor sufficient. 

\begin{conj}
There are infinitely many positive integers $n$ such that $sc_9(n)<sc_7(n)$.
\label{conj:9core}
\end{conj}

The choice of the lower bound $n\geq 56$ in the odd monotonicity conjecture was chosen because $sc_{11}(n)=sc_{9}(n)$ when $n\in\{0, 1, 2, 3, 4, 5, 6, 7, 8, 12, 14, 15, 16, 20, 22, 27, 31, 32, 35, 55\}$ and $sc_{11}(n)<sc_{9}(n)$ when $n$ equals 11 or 23. That these are the only values satisfying $sc_{11}(n)\leq sc_{9}(n)$ has been verified for all non-negative $n\leq 10000$.

\section{Future directions}
\label{sec:other}

In addition to the conjectures stated above, we have assembled multiple avenues for future exploration.

\subsection{Non-self-conjugate $t$-core partitions}\

To find a stronger monotonicity result for defect zero blocks of the alternating group, one would need to understand non-self-conjugate $t$-core partitions $nsc_t(n)$ as well.  Defect zero $t$-blocks arise in two ways.  The ones from $S_n$ that split upon restriction are counted by $2sc_t(n)$ and those from $S_n$ that do not split upon restriction are counted by $\frac{1}{2}nsc_t(n)$.  Experimentally, $nsc_{2t+3}(n)>nsc_{2t+1}(n)$ for $5\leq 2t+3\leq n\leq 500$, so we conjecture the following.

\begin{conj} Suppose $p,q$ are odd primes such that $9<p,q< n-17$. The number of defect zero $p$-blocks of $A_n$ is strictly less than the number of defect zero $q$-blocks of $A_n$.
\end{conj}

\subsection{Asymptotics and unimodality in self-conjugate core partitions}\

A deeper question than the monotonicity of $sc_{t+2}(n)>sc_{t}(n)$ has to do with the distribution of $sc_{t+2}(n)-sc_t(n)$ for a fixed $n$, and as $n$ goes to infinity.  

Define the functions $\pi_t(n)=(c_{t+1}(n)-c_{t}(n))/p(n)$ and $\sigma_t(n)=(sc_{t+2}(n)-sc_{t}(n))/sc(n)$ which are the normalized net increase in the number of partitions of $n$ that are $(t+1)$-cores and not $t$-cores and the normalized net increase in the number of self-conjugate core partitions of $n$ that are $(t+2)$-cores and not $t$-cores.  For fixed $n$, we can see that 
\[\sum_{t=1}^\infty \pi_t(n)=\sum_{t'=0}^{\infty} \sigma_{2t'}(n)=\sum_{t'=0}^{\infty} \sigma_{2t'+1}(n)=1.\] 
Plotting the functions (of $t$), $\pi_t(n)$, $\sigma_{2t}(n)$, and $\sigma_{2t+1}(n)$, for fixed values of $n$ between $100$ and $400$ gives the graphs in Figure~\ref{fig:distr}.    
\begin{figure}%
\hspace{-.3in}\epsfig{figure=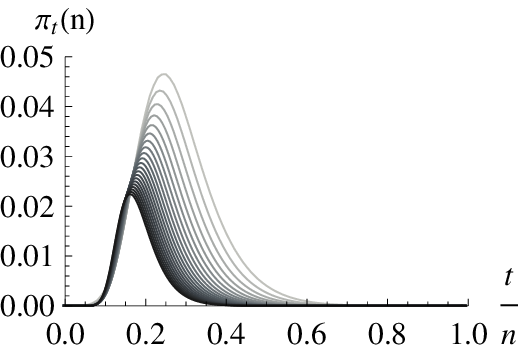,height=1.4in}~%
\epsfig{figure=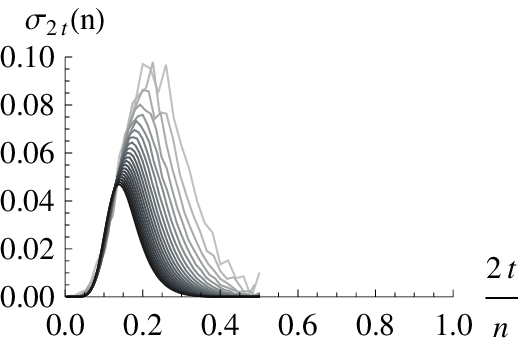,height=1.4in}~%
\epsfig{figure=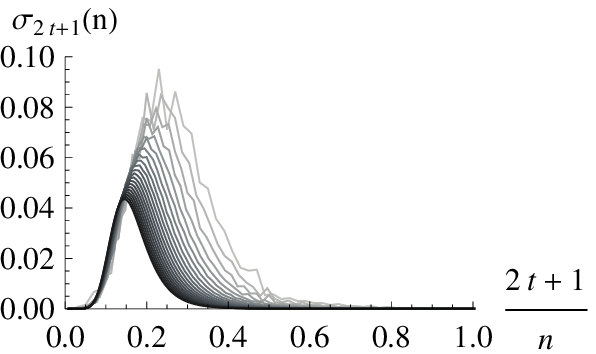,height=1.4in}\!\!\!\!\!\!\!\!%
\caption{Graphs of $\pi_t(n)$, $\sigma_{2t}(n)$, and $\sigma_{2t+1}(n)$ for values of $n$ between $100$ and $400$ (colored from light to dark).}
\label{fig:distr}
\end{figure}

In \cite{Craven}, Craven proves the following theorem.
\begin{thrm}[Craven]
Suppose that $0<q<1$ is a real number. Then as $n$ tends to infinity, $\displaystyle \frac{c_{\lfloor qn\rfloor}(n)}{p(n)}\rightarrow 1$.
\end{thrm}
As a consequence, as $n$ goes to infinity, $\pi_t(n)$ approaches the function that is identically zero.  This is seen in Figure~\ref{fig:distr}(a) by noticing that the function values in the sequence of curves at a fixed value on the $x$-axis eventually decreases to zero.  This appears to be true for self-conjugate partitions as well.

\begin{conj}
Suppose that $0<q<1$ is a real number. Then as $n$ tends to infinity, $\displaystyle \frac{sc_{\lfloor qn\rfloor}(n)}{sc(n)}\rightarrow 1$.
\end{conj}

It appears that much more is true.  Recall that a sequence $\{x_t\}_{0\leq t\leq r}$ is {\em unimodal} if there exists a number $T$ such that \[x_0\leq x_1 \leq \cdots x_{T-1} \leq x_T \geq x_{T+1} \geq x_{T+2} \geq \cdots \geq x_r.\]
Unimodality is a property that arises naturally in many areas, including combinatorics, geometry, and algebra; Brenti's survey article \cite{Unimo} gives examples and references. In \cite{Stanton}, Stanton discusses the unimodality of the coefficients of the generating function for partitions and self-conjugate partitions whose Young diagrams fit inside a given shape. 

It appears that for $n$ fixed and large enough, the sequences $\pi_t(n)$, $\sigma_{2t}(n)$, and $\sigma_{2t+1}(n)$ are unimodal.  We state these as conjectures.
\begin{conj}
For fixed $n\geq 63$, the sequence $\{\pi_t(n)\}_{4\leq t\leq n-7}$ is unimodal.
\end{conj}
\begin{conj}
\label{conj:unimodal}
For fixed $n\geq 139$, the sequence $\{\sigma_{2t}(n)\}_{8\leq 2t\leq 2\lfloor \frac{n}{4}\rfloor-8}$ is unimodal.  Further, for fixed $n\geq 213$, the sequence $\{\sigma_{2t+1}(n)\}_{9\leq 2t+1\leq \lfloor \frac{n}{2}\rfloor}$ is unimodal.
\end{conj}
The formulas given in Propositions~\ref{prop:evenLarge} and \ref{prop:oddLarge} allow for partial results toward Conjecture~\ref{conj:unimodal}, but the hard work is yet to be done.  

More pointedly, we can ask for the shape of the distribution---perhaps it is approaching a normal distribution, but after its peak it appears to decrease with a tail that is fatter than normal.  Because the pointwise limit of the distribution is the zero distribution (by Craven's theorem), the ``right question'' is more along the lines of finding the shape of the distribution as $n$ goes to infinity.  We state this as an open question.

\begin{question}
For $n$ sufficiently large, is there a limiting shape of the distributions of $\pi_t(n)$, $\sigma_{2t}(n)$, and $\sigma_{2t+1}(n)$?
\end{question}

Ideally, one would be able to find a combinatorial interpretation for $sc_{t+2}(n)-sc_t(n)$ to prove its positivity and understand its asymptotics.  

\begin{question}
Is there a simple combinatorial description of $c_{t+1}(n)-c_t(n)$?  Of $sc_{t+2}(n)-sc_t(n)$?
\end{question}

\subsection{Numerical identities and inequalities}\

Another direction is related to numerical identities involving core partitions. Garvan, Kim, and Stanton prove that $sc_5(2n+1) = sc_5(n)$, $sc_5(5n+4)=sc_5(n)$, and $sc_7(4n+6)=sc_7(n)$ using \cite[Equation~(7.4)]{Garvan}. Using Ramanujan's theta functions, Baruah and Berndt \cite{BB} prove $sc_3(4n+1) = sc_3(n)$ and Sarmah \cite{Sarmah-thesis} proves $sc_9(8n + 10) = sc_9(2n)$.  Further, Berkovich and Yesilyurt \cite{BY} prove inequalities such as $c_7(2n+2)\geq 2\,c_7(n)$ and $c_7(4n+6)\geq 10\,c_7(n)$.  

We aimed to find similar identities and inequalities.  Experimental data suggests the following conjectures.

\begin{conj} \label{conj:1.9} Let $n$ be a non-negative integer.  
\begin{enumerate}
\item Suppose $n\geq 49$.  Then $sc_9(4n)>3\,sc_9(n)$.
\item Suppose $n\geq 1$.  Then $sc_9(4n+1)>1.9\,sc_9(n)$.
\item Suppose $n\geq 17$.  Then $sc_9(4n+3)>1.9\,sc_9(n)$.
\item Suppose $n\geq 1$.  Then $sc_9(4n+4)>2.6\,sc_9(n)$.
\end{enumerate}
\end{conj}
Conjecture~\ref{conj:1.9} gives some conjectures in a family of inequalities of the form $sc_t(an+b)>\alpha\,sc_t(n)$.  It appears that for $t=9$ and $a=4$, then there exists a constant $\alpha>1$ where this is true for all $b$ not equal to $2$ modulo $4$.  It would be of interest to determine the value and interpretation of these constants.

We do not expect identities of the form $sc_t(an+b)=sc_t(n)$ for integers $a$ and $b$ for odd $t\geq 11$ and even $t\geq 8$, nor do we expect inequalities of the form $sc_t(an+b)>\alpha\,sc_t(n)$ for odd $t\leq 7$ and even $t\leq 6$.

\appendix
\section{Tables of values}
\label{sec:appendix}

Here we present tables of values of $sc_t(n)$ and $sc_{t+2}(n)-sc_t(n)$, generated by extracting coefficients from the generating function in Equation~\eqref{eq:scgf}.

\bigskip
{\bf Acknowledgements.} Both authors would like to thank Matt Fayers, Ben Ford, and Ken Ono for conversations on this topic.  We thank Noam Elkies for his conversations about the positivity of 7-core partitions and the correspondence with representations of integers as the sum of three squares.  We also would like to thank Nayandeep Deka Baruah, J\o rn Olsson, and an anonymous referee for correcting and suggesting references related to our work.  The second author thanks Jean-Baptiste Gramain for his hospitality while visiting the University of Paris and for early discussions on a Stanton-type conjecture for self-conjugate partitions while there.

\begin{sidewaystable}
\vspace{6.5in}
\hspace{-.5in}
{\tiny
\begin{tabular}{c|c@{\spa}c@{\spa}c@{\spa}c@{\spa}c@{\spa}c@{\spa}c@{\spa}c@{\spa}c@{\spa}c@{\spa}c@{\spa}c@{\spa}c@{\spa}c@{\spa}c@{\spa}c@{\spa}c@{\spa}c@{\spa}c@{\spa}c@{\spa}c@{\spa}c@{\spa}c@{\spa}c@{\spa}c@{\spa}c@{\spa}c@{\spa}c@{\spa}c@{\spa}c@{\spa}c@{\spa}c@{\spa}c@{\spa}c@{\spa}c@{\spa}c@{\spa}c@{\spa}c@{\spa}c@{\spa}c@{\spa}c@{\spa}c@{\spa}c@{\spa}c@{\spa}c@{\spa}c@{\spa}c@{\spa}c@{\spa}c@{\spa}c@{\spa}c@{\spa}c@{\spa}c@{\spa}c@{\spa}c@{\spa}c@{\spa}c@{\spa}c@{\spa}c@{\spa}c@{\spa}c@{\spa}} 
$_t\backslash^n$ & 0 &1 &2 &3& 4&5&6&7&8&9&10&11&12&13&14&15&16&17&18&19&20&21&22&23&24&25&26&27&28&29&30&31&32&33&34&35&36&37&38&39&40&41&42&43&44&45&46&47&48&49&50&51&52&53&54&55&56&57&58&59&60\\ \hline
2 & 1 & 1 & 0 & 1 & 0 & 0 & 1 & 0 & 0 & 0 & 1 & 0 & 0 & 0 & 0 & 1 & 0 & 0 & 0 & 0 & 0 & 1 & 0 & 0 & 0 & 0 & 0 & 0 & 1 & 0 & 0 & 0 & 0 & 0 & 0 & 0 & 1 & 0 & 0 & 0 & 0 & 0 & 0 & 0 & 0 & 1 & 0 & 0 & 0 & 0 & 0 & 0 & 0 & 0 & 0 & 1 & 0 & 0 & 0 & 0 & 0 \\
3 & \text{} & 1 & 0 & 0 & 0 & 1 & 0 & 0 & 1 & 0 & 0 & 0 & 0 & 0 & 0 & 0 & 1 & 0 & 0 & 0 & 0 & 1 & 0 & 0 & 0 & 0 & 0 & 0 & 0 & 0 & 0 & 0 & 0 & 1 & 0 & 0 & 0 & 0 & 0 & 0 & 1 & 0 & 0 & 0 & 0 & 0 & 0 & 0 & 0 & 0 & 0 & 0 & 0 & 0 & 0 & 0 & 1 & 0 & 0 & 0 & 0 \\
4 &  \text{} & \text{} & 0 & 1 & 1 & 1 & 1 & 1 & 0 & 0 & 2 & 0 & 1 & 1 & 1 & 2 & 0 & 0 & 1 & 1 & 0 & 1 & 1 & 0 & 1 & 2 & 0 & 2 & 1 & 0 & 1 & 0 & 1 & 1 & 1 & 0 & 1 & 0 & 0 & 1 & 3 & 1 & 0 & 1 & 0 & 2 & 1 & 0 & 1 & 1 & 1 & 0 & 1 & 0 & 0 & 2 & 0 & 1 & 0 & 1 & 2 \\
5 & \text{} & \text{} & \text{} & 1 & 1 & 0 & 0 & 1 & 1 & 1 & 0 & 0 & 2 & 0 & 0 & 1 & 2 & 1 & 0 & 1 & 0 & 0 & 0 & 0 & 1 & 2 & 0 & 0 & 2 & 0 & 0 & 1 & 0 & 2 & 0 & 1 & 2 & 0 & 0 & 1 & 2 & 0 & 0 & 0 & 1 & 0 & 0 & 0 & 1 & 1 & 0 & 2 & 2 & 0 & 0 & 0 & 0 & 2 & 0 & 0 & 2 \\ \hline
6 & \text{} & \text{} & \text{} & \text{} & 1 & 1 & 1 & 1 & 2 & 2 & 2 & 2 & 0 & 0 & 3 & 1 & 2 & 2 & 2 & 3 & 1 & 2 & 2 & 3 & 2 & 3 & 3 & 2 & 1 & 2 & 3 & 2 & 2 & 1 & 2 & 2 & 5 & 4 & 1 & 4 & 3 & 3 & 3 & 2 & 4 & 3 & 4 & 1 & 3 & 3 & 2 & 4 & 4 & 3 & 2 & 3 & 3 & 4 & 2 & 3 & 3 \\
7 & \text{} & \text{} & \text{} & \text{} & \text{} & 1 & 1 & 0 & 1 & 2 & 1 & 1 & 2 & 2 & 0 & 0 & 3 & 1 & 1 & 1 & 2 & 4 & 1 & 0 & 3 & 4 & 1 & 2 & 2 & 2 & 1 & 0 & 2 & 3 & 0 & 2 & 5 & 2 & 1 & 0 & 3 & 2 & 2 & 3 & 2 & 6 & 1 & 0 & 5 & 2 & 1 & 4 & 4 & 2 & 2 & 0 & 4 & 5 & 2 & 1 & 6 \\
8 & \text{} & \text{} & \text{} & \text{} & \text{} & \text{} & 1 & 1 & 2 & 2 & 2 & 2 & 3 & 3 & 3 & 4 & 1 & 1 & 5 & 2 & 3 & 4 & 4 & 5 & 3 & 4 & 4 & 6 & 4 & 5 & 6 & 4 & 5 & 7 & 6 & 7 & 7 & 5 & 7 & 7 & 6 & 5 & 8 & 5 & 5 & 6 & 6 & 6 & 13 & 11 & 4 & 11 & 7 & 9 & 9 & 6 & 11 & 12 & 10 & 8 & 13 \\
9 & \text{} & \text{} & \text{} & \text{} & \text{} & \text{} & \text{} & 1 & 2 & 1 & 1 & 2 & 2 & 2 & 2 & 3 & 4 & 3 & 0 & 1 & 5 & 2 & 2 & 3 & 4 & 4 & 1 & 5 & 6 & 4 & 3 & 5 & 7 & 4 & 1 & 6 & 8 & 5 & 3 & 5 & 8 & 5 & 2 & 5 & 8 & 6 & 4 & 4 & 11 & 6 & 1 & 5 & 9 & 5 & 6 & 11 & 8 & 9 & 2 & 8 & 14 \\
10 & \text{} & \text{} & \text{} & \text{} & \text{} & \text{} & \text{} & \text{} & 2 & 2 & 2 & 2 & 3 & 3 & 3 & 4 & 5 & 5 & 5 & 6 & 2 & 3 & 8 & 4 & 6 & 7 & 7 & 9 & 6 & 7 & 8 & 10 & 8 & 10 & 11 & 9 & 8 & 10 & 12 & 11 & 16 & 14 & 12 & 17 & 13 & 13 & 17 & 13 & 17 & 18 & 18 & 17 & 17 & 15 & 18 & 19 & 17 & 18 & 18 & 17 & 24 \\ \hline
11 & \text{} & \text{} & \text{} & \text{} & \text{} & \text{} & \text{} & \text{} & \text{} & 2 & 2 & 1 & 2 & 3 & 2 & 3 & 4 & 4 & 4 & 4 & 5 & 6 & 2 & 2 & 8 & 5 & 4 & 5 & 7 & 8 & 4 & 5 & 7 & 12 & 8 & 6 & 12 & 10 & 6 & 8 & 12 & 13 & 10 & 8 & 14 & 17 & 8 & 8 & 18 & 17 & 10 & 10 & 18 & 18 & 12 & 11 & 20 & 19 & 12 & 11 & 22 \\
12 & \text{} & \text{} & \text{} & \text{} & \text{} & \text{} & \text{} & \text{} & \text{} & \text{} & 2 & 2 & 3 & 3 & 3 & 4 & 5 & 5 & 5 & 6 & 7 & 8 & 8 & 9 & 5 & 6 & 12 & 8 & 10 & 11 & 12 & 14 & 11 & 13 & 14 & 17 & 15 & 17 & 19 & 17 & 16 & 19 & 22 & 21 & 21 & 20 & 24 & 24 & 30 & 30 & 26 & 32 & 30 & 32 & 34 & 33 & 37 & 36 & 40 & 36 & 38 \\
13 & \text{} & \text{} & \text{} & \text{} & \text{} & \text{} & \text{} & \text{} & \text{} & \text{} & \text{} & 2 & 3 & 2 & 2 & 4 & 4 & 4 & 4 & 5 & 6 & 6 & 6 & 7 & 9 & 9 & 4 & 6 & 12 & 7 & 8 & 10 & 12 & 13 & 8 & 11 & 14 & 14 & 10 & 18 & 21 & 13 & 14 & 18 & 22 & 19 & 16 & 20 & 26 & 23 & 16 & 22 & 33 & 26 & 20 & 32 & 35 & 30 & 24 & 28 & 43 \\
14 & \text{} & \text{} & \text{} & \text{} & \text{} & \text{} & \text{} & \text{} & \text{} & \text{} & \text{} & \text{} & 3 & 3 & 3 & 4 & 5 & 5 & 5 & 6 & 7 & 8 & 8 & 9 & 11 & 12 & 12 & 14 & 9 & 10 & 18 & 13 & 16 & 18 & 19 & 22 & 19 & 21 & 23 & 27 & 25 & 28 & 31 & 29 & 28 & 33 & 37 & 36 & 38 & 37 & 42 & 44 & 40 & 41 & 49 & 46 & 59 & 63 & 52 & 66 & 62 \\
15 & \text{} & \text{} & \text{} & \text{} & \text{} & \text{} & \text{} & \text{} & \text{} & \text{} & \text{} & \text{} & \text{} & 3 & 3 & 3 & 4 & 5 & 4 & 5 & 6 & 7 & 7 & 7 & 9 & 10 & 10 & 11 & 13 & 14 & 8 & 9 & 18 & 14 & 14 & 16 & 19 & 21 & 16 & 18 & 22 & 25 & 20 & 23 & 28 & 31 & 27 & 25 & 37 & 36 & 32 & 31 & 39 & 44 & 36 & 36 & 46 & 47 & 40 & 42 & 62 \\ \hline
16 & \text{} & \text{} & \text{} & \text{} & \text{} & \text{} & \text{} & \text{} & \text{} & \text{} & \text{} & \text{} & \text{} & \text{} & 3 & 4 & 5 & 5 & 5 & 6 & 7 & 8 & 8 & 9 & 11 & 12 & 12 & 14 & 16 & 17 & 18 & 20 & 15 & 17 & 26 & 21 & 25 & 27 & 29 & 33 & 30 & 33 & 36 & 41 & 39 & 44 & 48 & 46 & 47 & 53 & 58 & 59 & 61 & 61 & 69 & 72 & 69 & 72 & 82 & 80 & 81 \\
17 & \text{} & \text{} & \text{} & \text{} & \text{} & \text{} & \text{} & \text{} & \text{} & \text{} & \text{} & \text{} & \text{} & \text{} & \text{} & 4 & 5 & 4 & 4 & 6 & 6 & 7 & 7 & 8 & 10 & 10 & 10 & 12 & 14 & 14 & 15 & 17 & 19 & 20 & 14 & 17 & 27 & 21 & 22 & 26 & 30 & 31 & 26 & 31 & 35 & 38 & 34 & 39 & 46 & 42 & 38 & 52 & 59 & 50 & 55 & 57 & 66 & 65 & 58 & 68 & 80 \\
18 & \text{} & \text{} & \text{} & \text{} & \text{} & \text{} & \text{} & \text{} & \text{} & \text{} & \text{} & \text{} & \text{} & \text{} & \text{} & \text{} & 5 & 5 & 5 & 6 & 7 & 8 & 8 & 9 & 11 & 12 & 12 & 14 & 16 & 17 & 18 & 20 & 23 & 25 & 26 & 29 & 24 & 26 & 37 & 32 & 37 & 40 & 43 & 48 & 45 & 50 & 54 & 60 & 60 & 66 & 71 & 71 & 72 & 80 & 88 & 90 & 94 & 96 & 106 & 111 & 110 \\
19 & \text{} & \text{} & \text{} & \text{} & \text{} & \text{} & \text{} & \text{} & \text{} & \text{} & \text{} & \text{} & \text{} & \text{} & \text{} & \text{} & \text{} & 5 & 5 & 5 & 6 & 8 & 7 & 8 & 10 & 11 & 11 & 12 & 14 & 15 & 16 & 17 & 20 & 22 & 22 & 24 & 28 & 30 & 23 & 26 & 38 & 33 & 35 & 38 & 43 & 48 & 42 & 46 & 54 & 59 & 54 & 60 & 68 & 67 & 64 & 71 & 82 & 90 & 88 & 82 & 103 \\ 
20 & \text{} & \text{} & \text{} & \text{} & \text{} & \text{} & \text{} & \text{} & \text{} & \text{} & \text{} & \text{} & \text{} & \text{} & \text{} & \text{} & \text{} & \text{} & 5 & 6 & 7 & 8 & 8 & 9 & 11 & 12 & 12 & 14 & 16 & 17 & 18 & 20 & 23 & 25 & 26 & 29 & 33 & 35 & 37 & 41 & 36 & 39 & 52 & 47 & 53 & 58 & 62 & 68 & 67 & 73 & 78 & 87 & 87 & 95 & 103 & 104 & 107 & 118 & 128 & 132 & 139 \\ \hline
21 & \text{} & \text{} & \text{} & \text{} & \text{} & \text{} & \text{} & \text{} & \text{} & \text{} & \text{} & \text{} & \text{} & \text{} & \text{} & \text{} & \text{} & \text{} & \text{} & 6 & 7 & 7 & 7 & 9 & 10 & 11 & 11 & 13 & 15 & 15 & 16 & 18 & 21 & 22 & 23 & 26 & 29 & 30 & 32 & 36 & 40 & 42 & 35 & 40 & 54 & 48 & 51 & 57 & 64 & 68 & 63 & 71 & 79 & 84 & 81 & 91 & 101 & 99 & 98 & 110 & 123 \\
22 & \text{} & \text{} & \text{} & \text{} & \text{} & \text{} & \text{} & \text{} & \text{} & \text{} & \text{} & \text{} & \text{} & \text{} & \text{} & \text{} & \text{} & \text{} & \text{} & \text{} & 7 & 8 & 8 & 9 & 11 & 12 & 12 & 14 & 16 & 17 & 18 & 20 & 23 & 25 & 26 & 29 & 33 & 35 & 37 & 41 & 46 & 49 & 52 & 57 & 52 & 57 & 72 & 67 & 76 & 82 & 87 & 96 & 95 & 103 & 111 & 122 & 124 & 135 & 145 & 148 & 154 \\
23 & \text{} & \text{} & \text{} & \text{} & \text{} & \text{} & \text{} & \text{} & \text{} & \text{} & \text{} & \text{} & \text{} & \text{} & \text{} & \text{} & \text{} & \text{} & \text{} & \text{} & \text{} & 8 & 8 & 8 & 10 & 12 & 11 & 13 & 15 & 16 & 17 & 18 & 21 & 23 & 24 & 26 & 30 & 32 & 33 & 36 & 41 & 44 & 46 & 50 & 55 & 60 & 53 & 57 & 75 & 71 & 74 & 81 & 90 & 97 & 93 & 101 & 112 & 122 & 119 & 129 & 144 \\
24 & \text{} & \text{} & \text{} & \text{} & \text{} & \text{} & \text{} & \text{} & \text{} & \text{} & \text{} & \text{} & \text{} & \text{} & \text{} & \text{} & \text{} & \text{} & \text{} & \text{} & \text{} & \text{} & 8 & 9 & 11 & 12 & 12 & 14 & 16 & 17 & 18 & 20 & 23 & 25 & 26 & 29 & 33 & 35 & 37 & 41 & 46 & 49 & 52 & 57 & 63 & 68 & 72 & 78 & 75 & 81 & 98 & 95 & 105 & 113 & 121 & 132 & 133 & 144 & 154 & 168 & 173 \\ 
25 & \text{} & \text{} & \text{} & \text{} & \text{} & \text{} & \text{} & \text{} & \text{} & \text{} & \text{} & \text{} & \text{} & \text{} & \text{} & \text{} & \text{} & \text{} & \text{} & \text{} & \text{} & \text{} & \text{} & 9 & 11 & 11 & 11 & 14 & 15 & 16 & 17 & 19 & 22 & 23 & 24 & 27 & 31 & 32 & 34 & 38 & 42 & 44 & 47 & 52 & 57 & 61 & 64 & 70 & 78 & 82 & 75 & 84 & 103 & 98 & 105 & 115 & 126 & 134 & 131 & 144 & 158 \\ \hline
26 & \text{} & \text{} & \text{} & \text{} & \text{} & \text{} & \text{} & \text{} & \text{} & \text{} & \text{} & \text{} & \text{} & \text{} & \text{} & \text{} & \text{} & \text{} & \text{} & \text{} & \text{} & \text{} & \text{} & \text{} & 11 & 12 & 12 & 14 & 16 & 17 & 18 & 20 & 23 & 25 & 26 & 29 & 33 & 35 & 37 & 41 & 46 & 49 & 52 & 57 & 63 & 68 & 72 & 78 & 87 & 93 & 98 & 107 & 104 & 112 & 133 & 131 & 144 & 155 & 165 & 179 & 183 \\
27 & \text{} & \text{} & \text{} & \text{} & \text{} & \text{} & \text{} & \text{} & \text{} & \text{} & \text{} & \text{} & \text{} & \text{} & \text{} & \text{} & \text{} & \text{} & \text{} & \text{} & \text{} & \text{} & \text{} & \text{} & \text{} & 12 & 12 & 13 & 15 & 17 & 17 & 19 & 22 & 24 & 25 & 27 & 31 & 33 & 35 & 38 & 43 & 46 & 48 & 52 & 58 & 63 & 66 & 71 & 79 & 85 & 89 & 96 & 105 & 113 & 107 & 116 & 140 & 138 & 146 & 157 & 172 \\
28 & \text{} & \text{} & \text{} & \text{} & \text{} & \text{} & \text{} & \text{} & \text{} & \text{} & \text{} & \text{} & \text{} & \text{} & \text{} & \text{} & \text{} & \text{} & \text{} & \text{} & \text{} & \text{} & \text{} & \text{} & \text{} & \text{} & 12 & 14 & 16 & 17 & 18 & 20 & 23 & 25 & 26 & 29 & 33 & 35 & 37 & 41 & 46 & 49 & 52 & 57 & 63 & 68 & 72 & 78 & 87 & 93 & 98 & 107 & 117 & 125 & 133 & 144 & 143 & 154 & 178 & 178 & 195 \\
29 & \text{} & \text{} & \text{} & \text{} & \text{} & \text{} & \text{} & \text{} & \text{} & \text{} & \text{} & \text{} & \text{} & \text{} & \text{} & \text{} & \text{} & \text{} & \text{} & \text{} & \text{} & \text{} & \text{} & \text{} & \text{} & \text{} & \text{} & 14 & 16 & 16 & 17 & 20 & 22 & 24 & 25 & 28 & 32 & 33 & 35 & 39 & 44 & 46 & 49 & 54 & 59 & 63 & 67 & 73 & 81 & 86 & 90 & 99 & 108 & 114 & 121 & 132 & 143 & 152 & 148 & 161 & 189 \\ \hline
30 & \text{} & \text{} & \text{} & \text{} & \text{} & \text{} & \text{} & \text{} & \text{} & \text{} & \text{} & \text{} & \text{} & \text{} & \text{} & \text{} & \text{} & \text{} & \text{} & \text{} & \text{} & \text{} & \text{} & \text{} & \text{} & \text{} & \text{} & \text{} & 16 & 17 & 18 & 20 & 23 & 25 & 26 & 29 & 33 & 35 & 37 & 41 & 46 & 49 & 52 & 57 & 63 & 68 & 72 & 78 & 87 & 93 & 98 & 107 & 117 & 125 & 133 & 144 & 157 & 168 & 178 & 192 & 194 \\
31 & \text{} & \text{} & \text{} & \text{} & \text{} & \text{} & \text{} & \text{} & \text{} & \text{} & \text{} & \text{} & \text{} & \text{} & \text{} & \text{} & \text{} & \text{} & \text{} & \text{} & \text{} & \text{} & \text{} & \text{} & \text{} & \text{} & \text{} & \text{} & \text{} & 17 & 18 & 19 & 22 & 25 & 25 & 28 & 32 & 34 & 36 & 39 & 44 & 47 & 50 & 54 & 60 & 65 & 68 & 73 & 82 & 88 & 92 & 100 & 109 & 117 & 124 & 133 & 145 & 156 & 164 & 176 & 192 \\
32 & \text{} & \text{} & \text{} & \text{} & \text{} & \text{} & \text{} & \text{} & \text{} & \text{} & \text{} & \text{} & \text{} & \text{} & \text{} & \text{} & \text{} & \text{} & \text{} & \text{} & \text{} & \text{} & \text{} & \text{} & \text{} & \text{} & \text{} & \text{} & \text{} & \text{} & 18 & 20 & 23 & 25 & 26 & 29 & 33 & 35 & 37 & 41 & 46 & 49 & 52 & 57 & 63 & 68 & 72 & 78 & 87 & 93 & 98 & 107 & 117 & 125 & 133 & 144 & 157 & 168 & 178 & 192 & 209 \\
33 & \text{} & \text{} & \text{} & \text{} & \text{} & \text{} & \text{} & \text{} & \text{} & \text{} & \text{} & \text{} & \text{} & \text{} & \text{} & \text{} & \text{} & \text{} & \text{} & \text{} & \text{} & \text{} & \text{} & \text{} & \text{} & \text{} & \text{} & \text{} & \text{} & \text{} & \text{} & 20 & 23 & 24 & 25 & 29 & 32 & 34 & 36 & 40 & 45 & 47 & 50 & 55 & 61 & 65 & 69 & 75 & 83 & 88 & 93 & 102 & 111 & 118 & 125 & 136 & 148 & 157 & 166 & 180 & 195 \\
34 & \text{} & \text{} & \text{} & \text{} & \text{} & \text{} & \text{} & \text{} & \text{} & \text{} & \text{} & \text{} & \text{} & \text{} & \text{} & \text{} & \text{} & \text{} & \text{} & \text{} & \text{} & \text{} & \text{} & \text{} & \text{} & \text{} & \text{} & \text{} & \text{} & \text{} & \text{} & \text{} & 23 & 25 & 26 & 29 & 33 & 35 & 37 & 41 & 46 & 49 & 52 & 57 & 63 & 68 & 72 & 78 & 87 & 93 & 98 & 107 & 117 & 125 & 133 & 144 & 157 & 168 & 178 & 192 & 209 \\ 
35 & \text{} & \text{} & \text{} & \text{} & \text{} & \text{} & \text{} & \text{} & \text{} & \text{} & \text{} & \text{} & \text{} & \text{} & \text{} & \text{} & \text{} & \text{} & \text{} & \text{} & \text{} & \text{} & \text{} & \text{} & \text{} & \text{} & \text{} & \text{} & \text{} & \text{} & \text{} & \text{} & \text{} & 25 & 26 & 28 & 32 & 35 & 36 & 40 & 45 & 48 & 51 & 55 & 61 & 66 & 70 & 75 & 84 & 90 & 94 & 102 & 112 & 120 & 127 & 137 & 149 & 160 & 169 & 181 & 197 \\ \hline
36 & \text{} & \text{} & \text{} & \text{} & \text{} & \text{} & \text{} & \text{} & \text{} & \text{} & \text{} & \text{} & \text{} & \text{} & \text{} & \text{} & \text{} & \text{} & \text{} & \text{} & \text{} & \text{} & \text{} & \text{} & \text{} & \text{} & \text{} & \text{} & \text{} & \text{} & \text{} & \text{} & \text{} & \text{} & 26 & 29 & 33 & 35 & 37 & 41 & 46 & 49 & 52 & 57 & 63 & 68 & 72 & 78 & 87 & 93 & 98 & 107 & 117 & 125 & 133 & 144 & 157 & 168 & 178 & 192 & 209 \\
37 & \text{} & \text{} & \text{} & \text{} & \text{} & \text{} & \text{} & \text{} & \text{} & \text{} & \text{} & \text{} & \text{} & \text{} & \text{} & \text{} & \text{} & \text{} & \text{} & \text{} & \text{} & \text{} & \text{} & \text{} & \text{} & \text{} & \text{} & \text{} & \text{} & \text{} & \text{} & \text{} & \text{} & \text{} & \text{} & 29 & 33 & 34 & 36 & 41 & 45 & 48 & 51 & 56 & 62 & 66 & 70 & 76 & 85 & 90 & 95 & 104 & 113 & 120 & 128 & 139 & 151 & 161 & 170 & 184 & 200 \\
38 & \text{} & \text{} & \text{} & \text{} & \text{} & \text{} & \text{} & \text{} & \text{} & \text{} & \text{} & \text{} & \text{} & \text{} & \text{} & \text{} & \text{} & \text{} & \text{} & \text{} & \text{} & \text{} & \text{} & \text{} & \text{} & \text{} & \text{} & \text{} & \text{} & \text{} & \text{} & \text{} & \text{} & \text{} & \text{} & \text{} & 33 & 35 & 37 & 41 & 46 & 49 & 52 & 57 & 63 & 68 & 72 & 78 & 87 & 93 & 98 & 107 & 117 & 125 & 133 & 144 & 157 & 168 & 178 & 192 & 209 \\
39 & \text{} & \text{} & \text{} & \text{} & \text{} & \text{} & \text{} & \text{} & \text{} & \text{} & \text{} & \text{} & \text{} & \text{} & \text{} & \text{} & \text{} & \text{} & \text{} & \text{} & \text{} & \text{} & \text{} & \text{} & \text{} & \text{} & \text{} & \text{} & \text{} & \text{} & \text{} & \text{} & \text{} & \text{} & \text{} & \text{} & \text{} & 35 & 37 & 40 & 45 & 49 & 51 & 56 & 62 & 67 & 71 & 76 & 85 & 91 & 96 & 104 & 114 & 122 & 129 & 139 & 152 & 163 & 172 & 185 & 201 \\
40 & \text{} & \text{} & \text{} & \text{} & \text{} & \text{} & \text{} & \text{} & \text{} & \text{} & \text{} & \text{} & \text{} & \text{} & \text{} & \text{} & \text{} & \text{} & \text{} & \text{} & \text{} & \text{} & \text{} & \text{} & \text{} & \text{} & \text{} & \text{} & \text{} & \text{} & \text{} & \text{} & \text{} & \text{} & \text{} & \text{} & \text{} & \text{} & 37 & 41 & 46 & 49 & 52 & 57 & 63 & 68 & 72 & 78 & 87 & 93 & 98 & 107 & 117 & 125 & 133 & 144 & 157 & 168 & 178 & 192 & 209 \\ \hline
41 & \text{} & \text{} & \text{} & \text{} & \text{} & \text{} & \text{} & \text{} & \text{} & \text{} & \text{} & \text{} & \text{} & \text{} & \text{} & \text{} & \text{} & \text{} & \text{} & \text{} & \text{} & \text{} & \text{} & \text{} & \text{} & \text{} & \text{} & \text{} & \text{} & \text{} & \text{} & \text{} & \text{} & \text{} & \text{} & \text{} & \text{} & \text{} & \text{} & 41 & 46 & 48 & 51 & 57 & 62 & 67 & 71 & 77 & 86 & 91 & 96 & 105 & 115 & 122 & 130 & 141 & 153 & 163 & 173 & 187 & 203 \\
42 & \text{} & \text{} & \text{} & \text{} & \text{} & \text{} & \text{} & \text{} & \text{} & \text{} & \text{} & \text{} & \text{} & \text{} & \text{} & \text{} & \text{} & \text{} & \text{} & \text{} & \text{} & \text{} & \text{} & \text{} & \text{} & \text{} & \text{} & \text{} & \text{} & \text{} & \text{} & \text{} & \text{} & \text{} & \text{} & \text{} & \text{} & \text{} & \text{} & \text{} & 46 & 49 & 52 & 57 & 63 & 68 & 72 & 78 & 87 & 93 & 98 & 107 & 117 & 125 & 133 & 144 & 157 & 168 & 178 & 192 & 209 \\
43 & \text{} & \text{} & \text{} & \text{} & \text{} & \text{} & \text{} & \text{} & \text{} & \text{} & \text{} & \text{} & \text{} & \text{} & \text{} & \text{} & \text{} & \text{} & \text{} & \text{} & \text{} & \text{} & \text{} & \text{} & \text{} & \text{} & \text{} & \text{} & \text{} & \text{} & \text{} & \text{} & \text{} & \text{} & \text{} & \text{} & \text{} & \text{} & \text{} & \text{} & \text{} & 49 & 52 & 56 & 62 & 68 & 71 & 77 & 86 & 92 & 97 & 105 & 115 & 123 & 131 & 141 & 154 & 165 & 174 & 187 & 204 \\
44 & \text{} & \text{} & \text{} & \text{} & \text{} & \text{} & \text{} & \text{} & \text{} & \text{} & \text{} & \text{} & \text{} & \text{} & \text{} & \text{} & \text{} & \text{} & \text{} & \text{} & \text{} & \text{} & \text{} & \text{} & \text{} & \text{} & \text{} & \text{} & \text{} & \text{} & \text{} & \text{} & \text{} & \text{} & \text{} & \text{} & \text{} & \text{} & \text{} & \text{} & \text{} & \text{} & 52 & 57 & 63 & 68 & 72 & 78 & 87 & 93 & 98 & 107 & 117 & 125 & 133 & 144 & 157 & 168 & 178 & 192 & 209 \\
45 & \text{} & \text{} & \text{} & \text{} & \text{} & \text{} & \text{} & \text{} & \text{} & \text{} & \text{} & \text{} & \text{} & \text{} & \text{} & \text{} & \text{} & \text{} & \text{} & \text{} & \text{} & \text{} & \text{} & \text{} & \text{} & \text{} & \text{} & \text{} & \text{} & \text{} & \text{} & \text{} & \text{} & \text{} & \text{} & \text{} & \text{} & \text{} & \text{} & \text{} & \text{} & \text{} & \text{} & 57 & 63 & 67 & 71 & 78 & 86 & 92 & 97 & 106 & 116 & 123 & 131 & 142 & 155 & 165 & 175 & 189 & 205 \\ \hline
46 & \text{} & \text{} & \text{} & \text{} & \text{} & \text{} & \text{} & \text{} & \text{} & \text{} & \text{} & \text{} & \text{} & \text{} & \text{} & \text{} & \text{} & \text{} & \text{} & \text{} & \text{} & \text{} & \text{} & \text{} & \text{} & \text{} & \text{} & \text{} & \text{} & \text{} & \text{} & \text{} & \text{} & \text{} & \text{} & \text{} & \text{} & \text{} & \text{} & \text{} & \text{} & \text{} & \text{} & \text{} & 63 & 68 & 72 & 78 & 87 & 93 & 98 & 107 & 117 & 125 & 133 & 144 & 157 & 168 & 178 & 192 & 209 \\
47 & \text{} & \text{} & \text{} & \text{} & \text{} & \text{} & \text{} & \text{} & \text{} & \text{} & \text{} & \text{} & \text{} & \text{} & \text{} & \text{} & \text{} & \text{} & \text{} & \text{} & \text{} & \text{} & \text{} & \text{} & \text{} & \text{} & \text{} & \text{} & \text{} & \text{} & \text{} & \text{} & \text{} & \text{} & \text{} & \text{} & \text{} & \text{} & \text{} & \text{} & \text{} & \text{} & \text{} & \text{} & \text{} & 68 & 72 & 77 & 86 & 93 & 97 & 106 & 116 & 124 & 132 & 142 & 155 & 166 & 176 & 189 & 206 \\
48 & \text{} & \text{} & \text{} & \text{} & \text{} & \text{} & \text{} & \text{} & \text{} & \text{} & \text{} & \text{} & \text{} & \text{} & \text{} & \text{} & \text{} & \text{} & \text{} & \text{} & \text{} & \text{} & \text{} & \text{} & \text{} & \text{} & \text{} & \text{} & \text{} & \text{} & \text{} & \text{} & \text{} & \text{} & \text{} & \text{} & \text{} & \text{} & \text{} & \text{} & \text{} & \text{} & \text{} & \text{} & \text{} & \text{} & 72 & 78 & 87 & 93 & 98 & 107 & 117 & 125 & 133 & 144 & 157 & 168 & 178 & 192 & 209 \\
49 & \text{} & \text{} & \text{} & \text{} & \text{} & \text{} & \text{} & \text{} & \text{} & \text{} & \text{} & \text{} & \text{} & \text{} & \text{} & \text{} & \text{} & \text{} & \text{} & \text{} & \text{} & \text{} & \text{} & \text{} & \text{} & \text{} & \text{} & \text{} & \text{} & \text{} & \text{} & \text{} & \text{} & \text{} & \text{} & \text{} & \text{} & \text{} & \text{} & \text{} & \text{} & \text{} & \text{} & \text{} & \text{} & \text{} & \text{} & 78 & 87 & 92 & 97 & 107 & 116 & 124 & 132 & 143 & 156 & 166 & 176 & 190 & 207 \\
50 & \text{} & \text{} & \text{} & \text{} & \text{} & \text{} & \text{} & \text{} & \text{} & \text{} & \text{} & \text{} & \text{} & \text{} & \text{} & \text{} & \text{} & \text{} & \text{} & \text{} & \text{} & \text{} & \text{} & \text{} & \text{} & \text{} & \text{} & \text{} & \text{} & \text{} & \text{} & \text{} & \text{} & \text{} & \text{} & \text{} & \text{} & \text{} & \text{} & \text{} & \text{} & \text{} & \text{} & \text{} & \text{} & \text{} & \text{} & \text{} & 87 & 93 & 98 & 107 & 117 & 125 & 133 & 144 & 157 & 168 & 178 & 192 & 209 \\ \hline
\end{tabular}
}
\caption{A table of values of $sc_t(n)$ for $0\leq n\leq 60$ and $2\leq t\leq n+2$.}
\end{sidewaystable} 

\begin{sidewaystable}
\vspace{6.5in}
{\tiny
\begin{tabular}{c|c@{\spa}c@{\spa}c@{\spa}c@{\spa}c@{\spa}c@{\spa}c@{\spa}c@{\spa}c@{\spa}c@{\spa}c@{\spa}c@{\spa}c@{\spa}c@{\spa}c@{\spa}c@{\spa}c@{\spa}c@{\spa}c@{\spa}c@{\spa}c@{\spa}c@{\spa}c@{\spa}c@{\spa}c@{\spa}c@{\spa}c@{\spa}c@{\spa}c@{\spa}c@{\spa}c@{\spa}c@{\spa}c@{\spa}c@{\spa}c@{\spa}c@{\spa}c@{\spa}c@{\spa}c@{\spa}c@{\spa}c@{\spa}c@{\spa}c@{\spa}c@{\spa}c@{\spa}c@{\spa}c@{\spa}c@{\spa}c@{\spa}c@{\spa}c@{\spa}c@{\spa}c@{\spa}c@{\spa}c@{\spa}c@{\spa}c@{\spa}c@{\spa}c@{\spa}c@{\spa}c@{\spa}} 
$_{t}\backslash^n$ & 4&5&6&7&8&9&10&11&12&13&14&15&16&17&18&19&20&21&22&23&24&25&26&27&28&29&30&31&32&33&34&35&36&37&38&39&40&41&42&43&44&45&46&47&48&49&50&51&52&53&54&55&56&57&58&59&60\\ \hline
$4\!-\!2$ &  0 & 0 & 0 & 0 & 2 & 2 & 0 & 2 & $-$1 & $-$1 & 2 & $-$1 & 2 & 2 & 1 & 2 & 1 & 1 & 1 & 3 & 1 & 1 & 3 & 0 & 0 & 2 & 2 & 2 & 1 & 0 & 1 & 2 & 4 & 4 & 1 & 3 & 0 & 2 & 3 & 1 & 4 & 1 & 3 & 1 & 2 & 2 & 1 & 4 & 3 & 3 & 2 & 1 & 3 & 3 & 2 & 2 & 1 \\
$6\!-\!4$ &  \text{} & \text{} & \text{} & \text{} & 0 & 0 & 0 & 0 & 3 & 3 & 0 & 3 & $-$1 & $-$1 & 3 & $-$1 & 2 & 2 & 2 & 2 & 1 & 1 & 1 & 4 & 3 & 3 & 3 & 2 & 3 & 6 & 4 & 5 & 2 & 1 & 6 & 3 & 3 & 2 & 5 & 3 & 1 & 3 & 2 & 5 & 10 & 8 & 2 & 7 & 3 & 6 & 7 & 3 & 8 & 8 & 8 & 5 & 10 \\
$8\!-\!6$ &  \text{} & \text{} & \text{} & \text{} & \text{} & \text{} & \text{} & \text{} & 0 & 0 & 0 & 0 & 4 & 4 & 0 & 4 & $-$1 & $-$1 & 4 & $-$1 & 3 & 3 & 3 & 3 & 2 & 2 & 2 & 6 & 3 & 3 & 5 & 2 & 1 & 5 & 5 & 4 & 10 & 9 & 4 & 12 & 8 & 7 & 11 & 7 & 4 & 7 & 14 & 6 & 10 & 6 & 9 & 13 & 6 & 6 & 8 & 9 & 11 \\
$10\!-\!8$ & \text{} & \text{} & \text{} & \text{} & \text{} & \text{} & \text{} & \text{} & \text{} & \text{} & \text{} & \text{} & 0 & 0 & 0 & 0 & 5 & 5 & 0 & 5 & $-$1 & $-$1 & 5 & $-$1 & 4 & 4 & 4 & 4 & 3 & 3 & 3 & 8 & 7 & 7 & 7 & 6 & 0 & 5 & 10 & 4 & 8 & 7 & 7 & 11 & 13 & 12 & 8 & 15 & 13 & 17 & 16 & 14 & 20 & 18 & 22 & 19 & 14 \\ 
$12\!-\!10$ & \text{} & \text{} & \text{} & \text{} & \text{} & \text{} & \text{} & \text{} & \text{} & \text{} & \text{} & \text{} & \text{} & \text{} & \text{} & \text{} & 0 & 0 & 0 & 0 & 6 & 6 & 0 & 6 & $-$1 & $-$1 & 6 & $-$1 & 5 & 5 & 5 & 5 & 4 & 4 & 4 & 10 & 9 & 9 & 9 & 8 & 7 & 13 & 13 & 12 & 8 & 7 & 16 & 12 & 10 & 9 & 15 & 13 & 22 & 27 & 12 & 30 & 24 \\ \hline
$14\!-\!12$ &  \text{} & \text{} & \text{} & \text{} & \text{} & \text{} & \text{} & \text{} & \text{} & \text{} & \text{} & \text{} & \text{} & \text{} & \text{} & \text{} & \text{} & \text{} & \text{} & \text{} & 0 & 0 & 0 & 0 & 7 & 7 & 0 & 7 & $-$1 & $-$1 & 7 & $-$1 & 6 & 6 & 6 & 6 & 5 & 5 & 5 & 12 & 11 & 11 & 11 & 10 & 9 & 16 & 16 & 15 & 21 & 20 & 20 & 26 & 10 & 9 & 30 & 14 & 19 \\
$16\!-\!14$ & \text{} & \text{} & \text{} & \text{} & \text{} & \text{} & \text{} & \text{} & \text{} & \text{} & \text{} & \text{} & \text{} & \text{} & \text{} & \text{} & \text{} & \text{} & \text{} & \text{} & \text{} & \text{} & \text{} & \text{} & 0 & 0 & 0 & 0 & 8 & 8 & 0 & 8 & $-$1 & $-$1 & 8 & $-$1 & 7 & 7 & 7 & 7 & 6 & 6 & 6 & 14 & 13 & 13 & 13 & 12 & 11 & 19 & 19 & 18 & 25 & 24 & 24 & 31 & 29 \\
$18\!-\!16$ &  \text{} & \text{} & \text{} & \text{} & \text{} & \text{} & \text{} & \text{} & \text{} & \text{} & \text{} & \text{} & \text{} & \text{} & \text{} & \text{} & \text{} & \text{} & \text{} & \text{} & \text{} & \text{} & \text{} & \text{} & \text{} & \text{} & \text{} & \text{} & 0 & 0 & 0 & 0 & 9 & 9 & 0 & 9 & $-$1 & $-$1 & 9 & $-$1 & 8 & 8 & 8 & 8 & 7 & 7 & 7 & 16 & 15 & 15 & 15 & 14 & 13 & 22 & 22 & 21 & 29 \\
$20\!-\!18$ & \text{} & \text{} & \text{} & \text{} & \text{} & \text{} & \text{} & \text{} & \text{} & \text{} & \text{} & \text{} & \text{} & \text{} & \text{} & \text{} & \text{} & \text{} & \text{} & \text{} & \text{} & \text{} & \text{} & \text{} & \text{} & \text{} & \text{} & \text{} & \text{} & \text{} & \text{} & \text{} & 0 & 0 & 0 & 0 & 10 & 10 & 0 & 10 & $-$1 & $-$1 & 10 & $-$1 & 9 & 9 & 9 & 9 & 8 & 8 & 8 & 18 & 17 & 17 & 17 & 16 & 15 \\ 
$22\!-\!20$ &  \text{} & \text{} & \text{} & \text{} & \text{} & \text{} & \text{} & \text{} & \text{} & \text{} & \text{} & \text{} & \text{} & \text{} & \text{} & \text{} & \text{} & \text{} & \text{} & \text{} & \text{} & \text{} & \text{} & \text{} & \text{} & \text{} & \text{} & \text{} & \text{} & \text{} & \text{} & \text{} & \text{} & \text{} & \text{} & \text{} & 0 & 0 & 0 & 0 & 11 & 11 & 0 & 11 & $-$1 & $-$1 & 11 & $-$1 & 10 & 10 & 10 & 10 & 9 & 9 & 9 & 20 & 19 \\ \hline
$24\!-\!22$ & \text{} & \text{} & \text{} & \text{} & \text{} & \text{} & \text{} & \text{} & \text{} & \text{} & \text{} & \text{} & \text{} & \text{} & \text{} & \text{} & \text{} & \text{} & \text{} & \text{} & \text{} & \text{} & \text{} & \text{} & \text{} & \text{} & \text{} & \text{} & \text{} & \text{} & \text{} & \text{} & \text{} & \text{} & \text{} & \text{} & \text{} & \text{} & \text{} & \text{} & 0 & 0 & 0 & 0 & 12 & 12 & 0 & 12 & $-$1 & $-$1 & 12 & $-$1 & 11 & 11 & 11 & 11 & 10 \\
$26\!-\!24$ & \text{} & \text{} & \text{} & \text{} & \text{} & \text{} & \text{} & \text{} & \text{} & \text{} & \text{} & \text{} & \text{} & \text{} & \text{} & \text{} & \text{} & \text{} & \text{} & \text{} & \text{} & \text{} & \text{} & \text{} & \text{} & \text{} & \text{} & \text{} & \text{} & \text{} & \text{} & \text{} & \text{} & \text{} & \text{} & \text{} & \text{} & \text{} & \text{} & \text{} & \text{} & \text{} & \text{} & \text{} & 0 & 0 & 0 & 0 & 13 & 13 & 0 & 13 & $-$1 & $-$1 & 13 & $-$1 & 12 \\
$28\!-\!26$ &  \text{} & \text{} & \text{} & \text{} & \text{} & \text{} & \text{} & \text{} & \text{} & \text{} & \text{} & \text{} & \text{} & \text{} & \text{} & \text{} & \text{} & \text{} & \text{} & \text{} & \text{} & \text{} & \text{} & \text{} & \text{} & \text{} & \text{} & \text{} & \text{} & \text{} & \text{} & \text{} & \text{} & \text{} & \text{} & \text{} & \text{} & \text{} & \text{} & \text{} & \text{} & \text{} & \text{} & \text{} & \text{} & \text{} & \text{} & \text{} & 0 & 0 & 0 & 0 & 14 & 14 & 0 & 14 & $-$1 \\
$30\!-\!28$ &  \text{} & \text{} & \text{} & \text{} & \text{} & \text{} & \text{} & \text{} & \text{} & \text{} & \text{} & \text{} & \text{} & \text{} & \text{} & \text{} & \text{} & \text{} & \text{} & \text{} & \text{} & \text{} & \text{} & \text{} & \text{} & \text{} & \text{} & \text{} & \text{} & \text{} & \text{} & \text{} & \text{} & \text{} & \text{} & \text{} & \text{} & \text{} & \text{} & \text{} & \text{} & \text{} & \text{} & \text{} & \text{} & \text{} & \text{} & \text{} & \text{} & \text{} & \text{} & \text{} & 0 & 0 & 0 & 0 & 15 \\
$32\!-\!30$ & \text{} & \text{} & \text{} & \text{} & \text{} & \text{} & \text{} & \text{} & \text{} & \text{} & \text{} & \text{} & \text{} & \text{} & \text{} & \text{} & \text{} & \text{} & \text{} & \text{} & \text{} & \text{} & \text{} & \text{} & \text{} & \text{} & \text{} & \text{} & \text{} & \text{} & \text{} & \text{} & \text{} & \text{} & \text{} & \text{} & \text{} & \text{} & \text{} & \text{} & \text{} & \text{} & \text{} & \text{} & \text{} & \text{} & \text{} & \text{} & \text{} & \text{} & \text{} & \text{} & \text{} & \text{} & \text{} & \text{} & 0 \\ \hline
\end{tabular}

\bigskip\bigskip
\begin{tabular}{c|c@{\spa}c@{\spa}c@{\spa}c@{\spa}c@{\spa}c@{\spa}c@{\spa}c@{\spa}c@{\spa}c@{\spa}c@{\spa}c@{\spa}c@{\spa}c@{\spa}c@{\spa}c@{\spa}c@{\spa}c@{\spa}c@{\spa}c@{\spa}c@{\spa}c@{\spa}c@{\spa}c@{\spa}c@{\spa}c@{\spa}c@{\spa}c@{\spa}c@{\spa}c@{\spa}c@{\spa}c@{\spa}c@{\spa}c@{\spa}c@{\spa}c@{\spa}c@{\spa}c@{\spa}c@{\spa}c@{\spa}c@{\spa}c@{\spa}c@{\spa}c@{\spa}c@{\spa}c@{\spa}c@{\spa}c@{\spa}c@{\spa}c@{\spa}c@{\spa}c@{\spa}c@{\spa}c@{\spa}c@{\spa}c@{\spa}c@{\spa}c@{\spa}c@{\spa}c@{\spa}c} 
$_{t}\backslash^n$ &3&4&5&6&7&8&9&10&11&12&13&14&15&16&17&18&19&20&21&22&23&24&25&26&27&28&29&30&31&32&33&34&35&36&37&38&39&40&41&42&43&44&45&46&47&48&49&50&51&52&53&54&55&56&57&58&59&60\\ \hline
$5\!-\!3$ & 1 & 1 & $-$1 & 0 & 1 & 0 & 1 & 0 & 0 & 2 & 0 & 0 & 1 & 1 & 1 & 0 & 1 & 0 & $-$1 & 0 & 0 & 1 & 2 & 0 & 0 & 2 & 0 & 0 & 1 & 0 & 1 & 0 & 1 & 2 & 0 & 0 & 1 & 1 & 0 & 0 & 0 & 1 & 0 & 0 & 0 & 1 & 1 & 0 & 2 & 2 & 0 & 0 & 0 & $-$1 & 2 & 0 & 0 & 2 \\
$7\!-\!5$ &  0 & 0 & 1 & 1 & $-$1 & 0 & 1 & 1 & 1 & 0 & 2 & 0 & $-$1 & 1 & 0 & 1 & 0 & 2 & 4 & 1 & 0 & 2 & 2 & 1 & 2 & 0 & 2 & 1 & $-$1 & 2 & 1 & 0 & 1 & 3 & 2 & 1 & $-$1 & 1 & 2 & 2 & 3 & 1 & 6 & 1 & 0 & 4 & 1 & 1 & 2 & 2 & 2 & 2 & 0 & 4 & 3 & 2 & 1 & 4 \\
$9\!-\!7$ & \text{} & \text{} & 0 & 0 & 1 & 1 & $-$1 & 0 & 1 & 0 & 0 & 2 & 3 & 1 & 2 & $-$1 & 0 & 3 & $-$2 & 1 & 3 & 1 & 0 & 0 & 3 & 4 & 2 & 2 & 5 & 5 & 1 & 1 & 4 & 3 & 3 & 2 & 5 & 5 & 3 & 0 & 2 & 6 & 0 & 3 & 4 & 6 & 4 & 0 & 1 & 5 & 3 & 4 & 11 & 4 & 4 & 0 & 7 & 8 \\
$11\!-\!9$ & \text{} & \text{} & \text{} & \text{} & 0 & 0 & 1 & 1 & $-$1 & 0 & 1 & 0 & 0 & 0 & 1 & 4 & 3 & 0 & 4 & 0 & $-$1 & 4 & 1 & 3 & 0 & 1 & 4 & 1 & 0 & 0 & 8 & 7 & 0 & 4 & 5 & 3 & 3 & 4 & 8 & 8 & 3 & 6 & 11 & 4 & 4 & 7 & 11 & 9 & 5 & 9 & 13 & 6 & 0 & 12 & 10 & 10 & 3 & 8 \\ \hline
$13\!-\!11$ & \text{} & \text{} & \text{} & \text{} & \text{} & \text{} & 0 & 0 & 1 & 1 & $-$1 & 0 & 1 & 0 & 0 & 0 & 1 & 1 & 0 & 4 & 5 & 1 & 4 & 0 & 1 & 5 & $-$1 & 4 & 5 & 5 & 1 & 0 & 5 & 2 & 4 & 4 & 10 & 9 & 0 & 4 & 10 & 8 & 2 & 8 & 12 & 8 & 6 & 6 & 12 & 15 & 8 & 8 & 21 & 15 & 11 & 12 & 17 & 21 \\
$15\!-\!13$ & \text{} & \text{} & \text{} & \text{} & \text{} & \text{} & \text{} & \text{} & 0 & 0 & 1 & 1 & $-$1 & 0 & 1 & 0 & 0 & 0 & 1 & 1 & 0 & 0 & 1 & 6 & 5 & 1 & 7 & 0 & $-$1 & 6 & 1 & 6 & 5 & 5 & 7 & 6 & 0 & 1 & 12 & 6 & 5 & 6 & 12 & 11 & 5 & 11 & 13 & 16 & 9 & 6 & 18 & 16 & 4 & 11 & 17 & 16 & 14 & 19 \\
$17\!-\!15$ & \text{} & \text{} & \text{} & \text{} & \text{} & \text{} & \text{} & \text{} & \text{} & \text{} & 0 & 0 & 1 & 1 & $-$1 & 0 & 1 & 0 & 0 & 0 & 1 & 1 & 0 & 0 & 1 & 1 & 0 & 7 & 8 & 1 & 6 & 0 & 1 & 8 & 0 & 6 & 8 & 8 & 6 & 6 & 8 & 7 & 7 & 7 & 14 & 9 & 6 & 6 & 21 & 20 & 6 & 19 & 21 & 20 & 18 & 18 & 26 & 18 \\
$19\!-\!17$ &  \text{} & \text{} & \text{} & \text{} & \text{} & \text{} & \text{} & \text{} & \text{} & \text{} & \text{} & \text{} & 0 & 0 & 1 & 1 & $-$1 & 0 & 1 & 0 & 0 & 0 & 1 & 1 & 0 & 0 & 1 & 1 & 0 & 1 & 2 & 8 & 7 & 1 & 9 & 1 & 0 & 8 & 2 & 9 & 7 & 8 & 10 & 8 & 7 & 8 & 17 & 16 & 8 & 9 & 17 & 9 & 14 & 16 & 25 & 30 & 14 & 23 \\
$21\!-\!19$ & \text{} & \text{} & \text{} & \text{} & \text{} & \text{} & \text{} & \text{} & \text{} & \text{} & \text{} & \text{} & \text{} & \text{} & 0 & 0 & 1 & 1 & $-$1 & 0 & 1 & 0 & 0 & 0 & 1 & 1 & 0 & 0 & 1 & 1 & 0 & 1 & 2 & 1 & 0 & 9 & 10 & 2 & 9 & 0 & 2 & 11 & 0 & 9 & 11 & 10 & 9 & 9 & 11 & 11 & 17 & 17 & 20 & 19 & 9 & 10 & 28 & 20 \\ \hline
$23\!-\!21$ & \text{} & \text{} & \text{} & \text{} & \text{} & \text{} & \text{} & \text{} & \text{} & \text{} & \text{} & \text{} & \text{} & \text{} & \text{} & \text{} & 0 & 0 & 1 & 1 & $-$1 & 0 & 1 & 0 & 0 & 0 & 1 & 1 & 0 & 0 & 1 & 1 & 0 & 1 & 2 & 1 & 0 & 1 & 2 & 11 & 10 & 1 & 12 & 2 & 0 & 11 & 3 & 11 & 10 & 11 & 13 & 12 & 10 & 11 & 23 & 21 & 19 & 21 \\
$25\!-\!23$ & \text{} & \text{} & \text{} & \text{} & \text{} & \text{} & \text{} & \text{} & \text{} & \text{} & \text{} & \text{} & \text{} & \text{} & \text{} & \text{} & \text{} & \text{} & 0 & 0 & 1 & 1 & $-$1 & 0 & 1 & 0 & 0 & 0 & 1 & 1 & 0 & 0 & 1 & 1 & 0 & 1 & 2 & 1 & 0 & 1 & 2 & 2 & 1 & 11 & 13 & 3 & 11 & 1 & 3 & 13 & 1 & 12 & 14 & 14 & 12 & 12 & 15 & 14 \\
$27\!-\!25$ & \text{} & \text{} & \text{} & \text{} & \text{} & \text{} & \text{} & \text{} & \text{} & \text{} & \text{} & \text{} & \text{} & \text{} & \text{} & \text{} & \text{} & \text{} & \text{} & \text{} & 0 & 0 & 1 & 1 & $-$1 & 0 & 1 & 0 & 0 & 0 & 1 & 1 & 0 & 0 & 1 & 1 & 0 & 1 & 2 & 1 & 0 & 1 & 2 & 2 & 1 & 1 & 3 & 14 & 12 & 2 & 15 & 2 & 1 & 14 & 4 & 15 & 13 & 14 \\
$29\!-\!27$ & \text{} & \text{} & \text{} & \text{} & \text{} & \text{} & \text{} & \text{} & \text{} & \text{} & \text{} & \text{} & \text{} & \text{} & \text{} & \text{} & \text{} & \text{} & \text{} & \text{} & \text{} & \text{} & 0 & 0 & 1 & 1 & $-$1 & 0 & 1 & 0 & 0 & 0 & 1 & 1 & 0 & 0 & 1 & 1 & 0 & 1 & 2 & 1 & 0 & 1 & 2 & 2 & 1 & 1 & 3 & 3 & 1 & 14 & 16 & 3 & 14 & 2 & 4 & 17 \\
$31\!-\!29$ & \text{} & \text{} & \text{} & \text{} & \text{} & \text{} & \text{} & \text{} & \text{} & \text{} & \text{} & \text{} & \text{} & \text{} & \text{} & \text{} & \text{} & \text{} & \text{} & \text{} & \text{} & \text{} & \text{} & \text{} & 0 & 0 & 1 & 1 & $-$1 & 0 & 1 & 0 & 0 & 0 & 1 & 1 & 0 & 0 & 1 & 1 & 0 & 1 & 2 & 1 & 0 & 1 & 2 & 2 & 1 & 1 & 3 & 3 & 1 & 2 & 4 & 16 & 15 & 3 \\ \hline
$33\!-\!31$ & \text{} & \text{} & \text{} & \text{} & \text{} & \text{} & \text{} & \text{} & \text{} & \text{} & \text{} & \text{} & \text{} & \text{} & \text{} & \text{} & \text{} & \text{} & \text{} & \text{} & \text{} & \text{} & \text{} & \text{} & \text{} & \text{} & 0 & 0 & 1 & 1 & $-$1 & 0 & 1 & 0 & 0 & 0 & 1 & 1 & 0 & 0 & 1 & 1 & 0 & 1 & 2 & 1 & 0 & 1 & 2 & 2 & 1 & 1 & 3 & 3 & 1 & 2 & 4 & 3 \\
$35\!-\!33$ & \text{} & \text{} & \text{} & \text{} & \text{} & \text{} & \text{} & \text{} & \text{} & \text{} & \text{} & \text{} & \text{} & \text{} & \text{} & \text{} & \text{} & \text{} & \text{} & \text{} & \text{} & \text{} & \text{} & \text{} & \text{} & \text{} & \text{} & \text{} & 0 & 0 & 1 & 1 & $-$1 & 0 & 1 & 0 & 0 & 0 & 1 & 1 & 0 & 0 & 1 & 1 & 0 & 1 & 2 & 1 & 0 & 1 & 2 & 2 & 1 & 1 & 3 & 3 & 1 & 2 \\
$37\!-\!35$ & \text{} & \text{} & \text{} & \text{} & \text{} & \text{} & \text{} & \text{} & \text{} & \text{} & \text{} & \text{} & \text{} & \text{} & \text{} & \text{} & \text{} & \text{} & \text{} & \text{} & \text{} & \text{} & \text{} & \text{} & \text{} & \text{} & \text{} & \text{} & \text{} & \text{} & 0 & 0 & 1 & 1 & $-$1 & 0 & 1 & 0 & 0 & 0 & 1 & 1 & 0 & 0 & 1 & 1 & 0 & 1 & 2 & 1 & 0 & 1 & 2 & 2 & 1 & 1 & 3 & 3 \\
$39\!-\!37$ & \text{} & \text{} & \text{} & \text{} & \text{} & \text{} & \text{} & \text{} & \text{} & \text{} & \text{} & \text{} & \text{} & \text{} & \text{} & \text{} & \text{} & \text{} & \text{} & \text{} & \text{} & \text{} & \text{} & \text{} & \text{} & \text{} & \text{} & \text{} & \text{} & \text{} & \text{} & \text{} & 0 & 0 & 1 & 1 & $-$1 & 0 & 1 & 0 & 0 & 0 & 1 & 1 & 0 & 0 & 1 & 1 & 0 & 1 & 2 & 1 & 0 & 1 & 2 & 2 & 1 & 1 \\
$41\!-\!39$ & \text{} & \text{} & \text{} & \text{} & \text{} & \text{} & \text{} & \text{} & \text{} & \text{} & \text{} & \text{} & \text{} & \text{} & \text{} & \text{} & \text{} & \text{} & \text{} & \text{} & \text{} & \text{} & \text{} & \text{} & \text{} & \text{} & \text{} & \text{} & \text{} & \text{} & \text{} & \text{} & \text{} & \text{} & 0 & 0 & 1 & 1 & $-$1 & 0 & 1 & 0 & 0 & 0 & 1 & 1 & 0 & 0 & 1 & 1 & 0 & 1 & 2 & 1 & 0 & 1 & 2 & 2 \\ \hline
$43\!-\!41$ & \text{} & \text{} & \text{} & \text{} & \text{} & \text{} & \text{} & \text{} & \text{} & \text{} & \text{} & \text{} & \text{} & \text{} & \text{} & \text{} & \text{} & \text{} & \text{} & \text{} & \text{} & \text{} & \text{} & \text{} & \text{} & \text{} & \text{} & \text{} & \text{} & \text{} & \text{} & \text{} & \text{} & \text{} & \text{} & \text{} & 0 & 0 & 1 & 1 & $-$1 & 0 & 1 & 0 & 0 & 0 & 1 & 1 & 0 & 0 & 1 & 1 & 0 & 1 & 2 & 1 & 0 & 1 \\
$45\!-\!43$ & \text{} & \text{} & \text{} & \text{} & \text{} & \text{} & \text{} & \text{} & \text{} & \text{} & \text{} & \text{} & \text{} & \text{} & \text{} & \text{} & \text{} & \text{} & \text{} & \text{} & \text{} & \text{} & \text{} & \text{} & \text{} & \text{} & \text{} & \text{} & \text{} & \text{} & \text{} & \text{} & \text{} & \text{} & \text{} & \text{} & \text{} & \text{} & 0 & 0 & 1 & 1 & $-$1 & 0 & 1 & 0 & 0 & 0 & 1 & 1 & 0 & 0 & 1 & 1 & 0 & 1 & 2 & 1 \\
$47\!-\!45$ & \text{} & \text{} & \text{} & \text{} & \text{} & \text{} & \text{} & \text{} & \text{} & \text{} & \text{} & \text{} & \text{} & \text{} & \text{} & \text{} & \text{} & \text{} & \text{} & \text{} & \text{} & \text{} & \text{} & \text{} & \text{} & \text{} & \text{} & \text{} & \text{} & \text{} & \text{} & \text{} & \text{} & \text{} & \text{} & \text{} & \text{} & \text{} & \text{} & \text{} & 0 & 0 & 1 & 1 & $-$1 & 0 & 1 & 0 & 0 & 0 & 1 & 1 & 0 & 0 & 1 & 1 & 0 & 1 \\
$49\!-\!47$ & \text{} & \text{} & \text{} & \text{} & \text{} & \text{} & \text{} & \text{} & \text{} & \text{} & \text{} & \text{} & \text{} & \text{} & \text{} & \text{} & \text{} & \text{} & \text{} & \text{} & \text{} & \text{} & \text{} & \text{} & \text{} & \text{} & \text{} & \text{} & \text{} & \text{} & \text{} & \text{} & \text{} & \text{} & \text{} & \text{} & \text{} & \text{} & \text{} & \text{} & \text{} & \text{} & 0 & 0 & 1 & 1 & $-$1 & 0 & 1 & 0 & 0 & 0 & 1 & 1 & 0 & 0 & 1 & 1 \\
$51\!-\!49$ & \text{} & \text{} & \text{} & \text{} & \text{} & \text{} & \text{} & \text{} & \text{} & \text{} & \text{} & \text{} & \text{} & \text{} & \text{} & \text{} & \text{} & \text{} & \text{} & \text{} & \text{} & \text{} & \text{} & \text{} & \text{} & \text{} & \text{} & \text{} & \text{} & \text{} & \text{} & \text{} & \text{} & \text{} & \text{} & \text{} & \text{} & \text{} & \text{} & \text{} & \text{} & \text{} & \text{} & \text{} & 0 & 0 & 1 & 1 & $-$1 & 0 & 1 & 0 & 0 & 0 & 1 & 1 & 0 & 0 \\ \hline
$53\!-\!51$ & \text{} & \text{} & \text{} & \text{} & \text{} & \text{} & \text{} & \text{} & \text{} & \text{} & \text{} & \text{} & \text{} & \text{} & \text{} & \text{} & \text{} & \text{} & \text{} & \text{} & \text{} & \text{} & \text{} & \text{} & \text{} & \text{} & \text{} & \text{} & \text{} & \text{} & \text{} & \text{} & \text{} & \text{} & \text{} & \text{} & \text{} & \text{} & \text{} & \text{} & \text{} & \text{} & \text{} & \text{} & \text{} & \text{} & 0 & 0 & 1 & 1 & $-$1 & 0 & 1 & 0 & 0 & 0 & 1 & 1 \\
$55\!-\!53$ & \text{} & \text{} & \text{} & \text{} & \text{} & \text{} & \text{} & \text{} & \text{} & \text{} & \text{} & \text{} & \text{} & \text{} & \text{} & \text{} & \text{} & \text{} & \text{} & \text{} & \text{} & \text{} & \text{} & \text{} & \text{} & \text{} & \text{} & \text{} & \text{} & \text{} & \text{} & \text{} & \text{} & \text{} & \text{} & \text{} & \text{} & \text{} & \text{} & \text{} & \text{} & \text{} & \text{} & \text{} & \text{} & \text{} & \text{} & \text{} & 0 & 0 & 1 & 1 & $-$1 & 0 & 1 & 0 & 0 & 0 \\
$57\!-\!55$ & \text{} & \text{} & \text{} & \text{} & \text{} & \text{} & \text{} & \text{} & \text{} & \text{} & \text{} & \text{} & \text{} & \text{} & \text{} & \text{} & \text{} & \text{} & \text{} & \text{} & \text{} & \text{} & \text{} & \text{} & \text{} & \text{} & \text{} & \text{} & \text{} & \text{} & \text{} & \text{} & \text{} & \text{} & \text{} & \text{} & \text{} & \text{} & \text{} & \text{} & \text{} & \text{} & \text{} & \text{} & \text{} & \text{} & \text{} & \text{} & \text{} & \text{} & 0 & 0 & 1 & 1 & $-$1 & 0 & 1 & 0 \\
$59\!-\!57$ & \text{} & \text{} & \text{} & \text{} & \text{} & \text{} & \text{} & \text{} & \text{} & \text{} & \text{} & \text{} & \text{} & \text{} & \text{} & \text{} & \text{} & \text{} & \text{} & \text{} & \text{} & \text{} & \text{} & \text{} & \text{} & \text{} & \text{} & \text{} & \text{} & \text{} & \text{} & \text{} & \text{} & \text{} & \text{} & \text{} & \text{} & \text{} & \text{} & \text{} & \text{} & \text{} & \text{} & \text{} & \text{} & \text{} & \text{} & \text{} & \text{} & \text{} & \text{} & \text{} & 0 & 0 & 1 & 1 & $-$1 & 0 \\
$61\!-\!59$ & \text{} & \text{} & \text{} & \text{} & \text{} & \text{} & \text{} & \text{} & \text{} & \text{} & \text{} & \text{} & \text{} & \text{} & \text{} & \text{} & \text{} & \text{} & \text{} & \text{} & \text{} & \text{} & \text{} & \text{} & \text{} & \text{} & \text{} & \text{} & \text{} & \text{} & \text{} & \text{} & \text{} & \text{} & \text{} & \text{} & \text{} & \text{} & \text{} & \text{} & \text{} & \text{} & \text{} & \text{} & \text{} & \text{} & \text{} & \text{} & \text{} & \text{} & \text{} & \text{} & \text{} & \text{} & 0 & 0 & 1 & 1 \\ \hline
\end{tabular}
}
\caption{Table of values of $sc_{2t+2}(n)-sc_{2t}(n)$ and $sc_{2t+3}(n)-sc_{2t+1}(n)$.}
\end{sidewaystable}


\bibliographystyle{alpha}

\end{document}